\documentclass[letterpaper,10pt]{amsart}
\usepackage[utf8]{inputenc}
\usepackage{amsmath}
\usepackage{amssymb}
\usepackage{amsthm}
\usepackage{bbm}
\usepackage{geometry}
\usepackage{comment}
\usepackage{hyperref}
\usepackage{mathtools}
\usepackage{stmaryrd}
\usepackage{enumerate}

\usepackage[title]{appendix}
\newcommand*{\mailto}[1]{\href{mailto:#1}{\nolinkurl{#1}}}
\newcommand{\arxiv}[1]{\href{http://arxiv.org/abs/#1}{arXiv:#1}}
\newcommand{\doi}[1]{\href{http://dx.doi.org/#1}{DOI:#1}}

\hypersetup{
    colorlinks,
    citecolor=blue,
    filecolor=blue,
    linkcolor=blue,
    urlcolor=blue
}
\usepackage{appendix}
\usepackage{graphicx}
\usepackage{csquotes}
\newtheorem{thm}{Theorem}[section]
\newtheorem{prop}[thm]{Proposition}

\newtheorem{prob}[thm]{Open Problem}
\newtheorem{lem}[thm]{Lemma}
\newtheorem{definition}[thm]{Definition}

\newtheorem{remark}[thm]{Remark}
\newtheorem{corollary}[thm]{Corollary}

\newcommand{\N}{\mathbb N}
\newcommand{\Z}{\mathbb Z}

\newcommand{\R}{\mathbb R}

\DeclareMathOperator{\Li}{Li}
\DeclareMathOperator{\sech}{sech}

\allowdisplaybreaks

\numberwithin{equation}{section}

\begin{document}

\author[B.\ Rosenzweig]{Bart Rosenzweig}
\address[B.\ Rosenzweig]{Department of Mathematics, The Ohio State University \\
100 Math Tower, 231 West 18th Avenue, Columbus, OH 43210, USA}
\email{\mailto{rosenzweig.26@osu.edu}}

\author[J.\ Stanfill]{Jonathan Stanfill}
\address[J.\ Stanfill]{Division of Geodetic Science, School of Earth Sciences, The Ohio State University \\
275 Mendenhall Laboratory, 125 South Oval Mall, Columbus, OH 43210, USA}
\email{\mailto{stanfill.13@osu.edu}}

\title[On the fundamental solutions of two nonlocal parabolic equations]{On the fundamental solutions of two nonlocal parabolic equations related to  logarithmic Laplacians}


\begin{abstract}
We answer in the affirmative a question posed by V. Maz'ya in \cite{Maz18} of whether one can continue as a meromorphic function of $t$ the series representation of the fundamental solution of a certain nonlocal parabolic equation associated to a logarithmic Laplacian, which arises in the study of boundary value problems associated to the ordinary Laplacian on domains with thin cavities, as in \cite{MNP00}. The $a\ln(n) +O(1)$ growth of the eigenvalues of the integral operator, together with explicit formulas for the eigenfunctions and the subleading asymptotic behavior of the eigenvalues, allows one to show that the fundamental solution is reminiscent of a sum of shifted Riemann zeta functions or polylogarithms, depending on the spatial variable. We show an analogous result for an operator also considered in \cite{MNP00}, related to a different logarithmic Laplacian, whose structure is similar. Along the way we are led to prove and to conjecture a number of curious identities involving Bell polynomials and Bernoulli numbers related to the exponential of the digamma function which are of independent interest.
\end{abstract}

\maketitle

\section{Introduction and main results}

We study the Cauchy problems associated to two integral operators, the first of which is given by
\begin{align}
    A_1h(e^{i\varphi}) := -\int_{C}\frac{h(z)-h(y)}{|z-y|}|dz| = -\frac{1}{2}\int_0^{2\pi}\frac{h(e^{i\tau})-h(e^{i\varphi})}{|\sin(\tfrac{\tau-\varphi}{2})|}d\varphi,
\end{align}
where $z=e^{i\varphi}$, $C$ is the unit circle in $\R^2$ that has been parameterized in the usual way, and where we note that $|z-y|$ represents the chordal distance between $z,y \in C$. In what follows, we denote by $K_1(t,\varphi,\theta)$ (the spectral series representation of) the fundamental solution of the associated Cauchy problem $\partial_t u(t,\varphi) = -A_1u(t,\varphi)$. In \cite[Sect. 10.1]{Maz18}, Maz'ya posed the problem of whether meromorphically continuing $K_1(t,\varphi,\theta)$ in the $t$ variable is possible, pointing out that $K_1$ is ``reminiscent'' of the modified Riemann zeta function $\sum_{n=1}^\infty e^{-tH_n} = e^{-\gamma_Et}\sum_{n=1}^\infty n^{-t}e^{-th_n}$, where $H_n:=\sum_{k=1}^n \frac{1}{k} = \ln(n) + \gamma_E +h_n$ denotes the $n$th harmonic number, $\gamma_E$ being the Euler-Mascheroni constant, and where $h_n = O(n^{-1})$ for $n\to \infty$. We answer this question in the affirmative in Theorem \ref{Thm:Main3}, the proof of which is contained in Section \ref{S2}.

The operator $A_1$, once defined on its natural domain, the closure of $C_c^\infty(C)$ with respect to the norm given by the induced quadratic form (cf. discussion at top of p. 59 in \cite{MNP00}), is known to have purely discrete spectrum with corresponding eigenvalue/eigenfunction pairs given by
\begin{align}
    A_1 e^{\pm i n\varphi} = 2\lambda_ne^{\pm i n\varphi},\quad n\in\mathbb{N}_0:=\mathbb{N}\cup \{0\},\ \varphi\in[0,2\pi),
\end{align}
where $\lambda_n:=2\sum_{j=0}^{n-1}(2j+1)^{-1} = 2H_{2n}-H_n$, and the empty sum for $n = 0$ is taken to be zero. Therefore, the fundamental solution to the above Cauchy problem has the convergent representation
\begin{align}\label{Eq:SpecSer}
    K_1(t,\varphi,\theta) = \frac{1}{2\pi} + \frac{1}{2\pi}\sum_{n \in \Z \setminus \{0\}} e^{in(\varphi-\theta)}e^{-2\lambda_{|n|}t},\quad \varphi,\theta\in[0,2\pi),\ \Re(t)>\frac{1}{2}\delta_{\varphi,\theta}\,\,,
\end{align}
where $\delta_{i,j}$ denotes the Kronecker delta function. This series representation of $K_1(t,\varphi,\theta)$ only converges pointwise for $\Re(t)>\frac{1}{2}$ on the diagonal $\varphi=\theta$ since $2\lambda_{n} = 2\ln(n) + O(1)$ as $n \rightarrow \infty$, and the eigenfunctions have modulus 1, yielding leading behavior $n^{-2t}$ of the summands. The pole which consequently emerges at $t=\frac{1}{2}$ represents a threshold for the compact operator $e^{-tA_1}$ being trace class; in fact, $e^{-tA_1}$ is trace class if and only if $\Re(t)>\frac{1}{2}$. 

We note that our operator $A_1$ is the same, up to a constant shift and a multiplicative factor, as the conformal logarithmic Laplacian on $\mathbb{S}^1$ introduced in \cite{FeS25} which is given by
\begin{align} \mathcal{P}_g^{\textnormal{log}} \, u(z):=\int_{C}  \frac{u(z)-u(y)}{|z-y|}dV_g + 2\psi(\tfrac{1}{2})u(z),
\end{align}
where $g$ is the round Riemannian metric on $\mathbb{S}^1$ and $\psi(\,\cdot\,)$ is the digamma function. This operator arises as the formal derivative at $s =0$ of the conformal fractional Laplacian on the sphere (see \cite{FeS25} for details), and can be studied in dimension $n$, where the spectrum is also known explicitly (the eigenvalues are given explicitly in terms of those of the Laplace-Beltrami operator on $\mathbb{S}^n$ equipped with $g$, and the eigenfunctions are the corresponding spherical harmonics); therefore the higher-dimensional analogues of $A_1$ may be amenable to analysis similar to what is done in this work, but we leave that as a potential future direction for now. 

We also note that the logarithmic Laplacian on all of $\R^n$ is by now well-studied, see for example \cite{CW19} where it was introduced, as well as \cite{LW21} and \cite{CV24}, the latter of which includes an in-depth study of the fundamental solution of the associated Cauchy problem.

In order to state our main results regarding this problem, we first introduce the following notation:

\begin{definition}\label{Def:pj}
Given a sequence of numbers, $S$, indexed over a set $J\supseteq \mathbb{N}$, we define
\begin{equation}\label{Eq:12a}
p_{j,S}(t):=\sum_{k=0}^j \frac{(-1)^k}{k!} \widehat{B}_{j,k}(s_1,\dots,s_{j-k+1})t^k=\frac{(-1)^j}{j!}\det\mathcal{N}_{j,S}(t),\quad j\in\mathbb{N}_0,
\end{equation}
where $\widehat{B}_{n,k}$ denote the \textup{partial ordinary Bell polynomials} (see Section \ref{S:Notation}) and the $j\times j$ lower Hessenberg matrix $\mathcal{N}_{j,S}(t)$ consists of $s_1 t$ on the diagonal, $(k+1)s_{k+1} t$ on the $k$th subdiagonal, the sequence $1,2,\dots, j-1$ on the first superdiagonal, and zeros on all other superdiagonals (cf. \cite[Eq. (5.3)]{OS22}):
\begin{align}\label{Eq:matrix}
    \mathcal{N}_{j,S}(t) = \begin{pmatrix}
        s_1 t && 1 && 0  && \cdots &&\cdots\\ 2s_2t && s_1t&&2&&0 &&\cdots \\ 
        3s_3t && 2s_2t&& s_1 t&& 3 &&\cdots \\ \vdots && &&  && \ddots \\ js_jt && (j-1)s_{j-1}t && \cdots &&  \cdots && s_1 t
    \end{pmatrix}.
\end{align}
\end{definition}

The polynomials $p_{j,S}(t)$ will appear throughout our analysis in our closed-form expressions for the residues. We refer to Sections \ref{S:Bernoulli} and \ref{Sect:poly} for identities  of independent interest involving these polynomials whenever the sequence, $S$, contains Bernoulli numbers.

We can now state our first main result.

\begin{thm}[Analytic structure of $K_1$]\label{Thm:Main3}
    For each $\varphi,\theta\in[0,2\pi)$, $\varphi\neq\theta$, the  solution $K_1(t,\varphi,\theta)$ extends to an entire function of $t$.

    For each $\varphi\in[0,2\pi)$, the solution $K_1(t,\varphi,\varphi)$ extends to a meromorphic function of $t$ with a simple pole at $t=\tfrac{1}{2}$ and  possible simple poles at $t=\frac{1}{2}-m$, $m\in\mathbb{N}$ with corresponding residues given by 
    \begin{equation}
        \frac{e^{-2\ln(2)-\gamma_E}}{2\pi}e^{[4\ln(2)+2\gamma_E] m}p_{2m,S_1}\big(\tfrac{1}{2}-m\big),\quad m\in\mathbb{N}_0,
    \end{equation}
    where $\gamma_E$ denotes the Euler-Mascheroni constant and, denoting by $B_\ell$ the Bernoulli numbers with $B_1=-1/2$ and by $B_k(s)$ the Bernoulli polynomials, the sequence $S_1$ satisfies
\begin{align}\label{Eq:S1}
s_k^{(1)}=(1-2^{1-k})\frac{2 B_{k}}{k}=-\frac{2}{k}B_k(1/2),\quad k\in\mathbb{N}.
\end{align} 

In particular, $t=0$ is a regular point of $K_1(t,\varphi,\theta)$ for all $\varphi,\theta\in[0,2\pi)$.
\end{thm}

We conjecture that $K_1(t,\varphi,\varphi)$ in fact has simple poles exactly at $t=\frac{1}{2}-m$, $m\in\mathbb{N}_0$, which reduces to showing the following, readily verified for small $m$ (see Remark \ref{Remark2} for equivalent reformulations):

\begin{prob}\label{Problem1}
Show that $p_{2m,S_1}(\tfrac{1}{2}-m)\neq0$ for all $m\in\mathbb{N}$ where the sequence $S_1$ satisfies \eqref{Eq:S1}. 
\end{prob}

\bigskip

We now turn to the second Cauchy problem we are interested in, $\partial_t u(t,x)=-A_2u(t,x)$,
where $A_2$ is the following integral operator considered in Section 12.5.4 of \cite{MNP00} (see also \cite{T64} where this integral operator seems to have first been considered in the context of slender-body potential theory, and the discussion in Section 5.3 of the more recent paper \cite{GM25})
\begin{align}
    A_2u(t,x):=\frac{1}{2}\int_{-1}^1 \frac{u(t,x)-u(t,z)}{|x-z|}\,dz.
\end{align}
Properties of the operator $A_1$ were needed in Section 12.2 of \cite{MNP00} in the construction of a boundary layer for the solution to the Poisson equation in the exterior of a tube in $\R^3$. Similar considerations in Section 12.5 of that reference led the authors to study $A_2$, which suggests going through the same analysis for this operator. We also note that, like $A_1$, the operator $A_2$ may be obtained as a formal derivative at $s=0$ of a fractional Laplacian-type operator on the interval $(-1,1)$ (in the parlance of \cite{GS19} a `regional' fractional Laplacian on the interval), that is, an operator of the form $c_s\int_{-1}^1 \frac{u(x)-u(z)}{|x-z|^{1+2s}}dz$ for an appropriate normalization constant $c_s$; or, as the logarithmic Dirichlet Laplacian on the interval $[-1,1]$, thought of as a metric-measure space, as in \cite{GM25}.

If the domain of $A_2$ is taken to be the closure of $C_c^\infty[-1,1]$ in the norm given by the induced quadratic form, then $A_2$ is known to have purely discrete spectrum with corresponding eigenvalue/eigenfunction pairs given by
\begin{align}
    A_2 P_n(x)=H_nP_n(x),\quad n\in\mathbb{N}_0,\ x\in[-1,1],
\end{align}
where $P_n(x)$ denotes the $n$th Legendre polynomial and $H_0 := 0$ once again. Therefore, the fundamental solution to the above Cauchy problem has the representation
\begin{align}\label{Eq:SpecSer2}
    K_2(t,x,y) = \frac{1}{2}+\sum_{n\in\mathbb{N}} \frac{2n+1}{2}P_n(x)P_n(y)e^{-tH_n},\quad x,y\in[-1,1],\ \Re(t)>0 \text{ sufficiently large}.
\end{align}
The exact region of convergence for $t$ can have lower bound $2,1,1/2,$ or $0$ depending on the choices of $x,y$:
\begin{equation}
\Re(t)>1+\delta_{x,y}\delta_{x,\pm1}-\tfrac{1}{2}(1-\delta_{x,\pm1}+1-\delta_{y,\pm1})(1-\delta_{x,y}),
\end{equation}
where $\delta_{i,j}$ denotes the Kronecker delta function and the region of convergence follows similarly to before by taking into account the leading term and oscillations of Hilb's formula for the asymptotics of $P_n(x)$ (cf. Thm. 8.21.2 in \cite{Sz39}).

Similarly to before, the threshold of the compact operator $e^{-t A_2}$ being trace class is related to the pointwise divergence in $t$ of the spectral series representation \eqref{Eq:SpecSer2} on the diagonal, but the relationship is a bit more subtle now due to the behavior of the normalized eigenfunctions at particular values of $x,y$ in the series which is not accounted for when considering the semigroup. In particular, one has that $e^{-t A_2}$ is trace class if and only if $\Re(t)>1$ (cf. \cite[Prop. 5.8]{GM25} noting our $A_2$ is $\frac{1}{2}$ times their integral operator), meaning the range coincides with the pointwise convergence of the series on the diagonal except for the edge cases $x=\pm1$.

The analytic continuation of this expression in $t$ is quite interesting and much more nontrivial than our first example, even though the result can be summarized similarly.
Moreover, it is in the same ways ``reminiscent'' of the previously mentioned modified Riemann zeta function given in Maz'ya's original problem statement.
As the resulting equations are much lengthier in this case, we include two theorems: the first considers the analytic structure while the second includes the residues explicitly.

\begin{thm}[Analytic structure of $K_2$]\label{Thm:Main}
    For each $x,y \in [-1,1]$, $x\neq y$, the solution $K_2(t,x,y)$ extends to an entire function of $t$. 

    For each $x \in [-1,1]$ the solution $K_2(t,x,x)$ extends to a meromorphic function of $t$ with a simple pole at $t=1+\delta_{x,\pm1}$ and additional possible simple poles at $t=-(1+\delta_{x,\pm1})m$, $m\in\mathbb{N}$.

    In particular, $t=0$ is a regular point of $K_2(t,x,y)$ for all $x,y\in[-1,1]$.
\end{thm}

That our theorems show that both $K_1(t,\varphi,\theta)$ and $K_2(t,x,y)$ are in fact entire in $t$ for all fixed spatial variable values off the diagonal may seem surprising at first glance from the series representations; however, it is not altogether surprising due to the oscillatory nature of the corresponding eigenfunctions. 
For instance, on the diagonal for $K_1$ as well as for the cases $x=y=\pm1$ for $K_2$, one obtains meromorphic extensions comprising of shifted Riemann zeta functions (cf. \eqref{Eq:29a}, \eqref{Eq:Mero}) plus entire remainder terms. One should contrast this to the off-diagonal for $K_1$ and the cases $x=-y=\pm1$ for $K_2$ which yield a similar expression, where the Riemann zeta function is replaced with the polylogarithm and the Dirichlet eta function, respectively (cf. \eqref{Eq:29aa}, \eqref{Eq:Ent}). Consequently, in these cases the functions are entire in $t$ simply due to the alternation introduced in the summation. When $x\neq y$ for $K_2$, the oscillatory nature of the Legendre polynomial terms as $n\to\infty$ should be expected to yield a similar structure to that of the previous case, that is, the function should be entire.

Finally, we point out that while one often expects a meromorphic continuation to exist in such contexts, that is not at all guaranteed. For instance, we refer the interested reader to the recent papers \cite{FPS25,FS25} of one of the authors regarding spectral $\zeta$-functions for examples where such series related to the spectrum of certain Schr\"odinger operators with Bessel, Legendre, and P\"oschl--Teller potentials do not admit meromorphic extensions. These examples instead have logarithmic branch points including infinitely many in the latter two cases (cf. \cite[Sect. 5.2]{FPS25} and \cite[Thm. 1.1]{FS25}, resp.).

We conclude this part of the introduction by discussing the residues at the poles in Theorem \ref{Thm:Main}.

\begin{thm}[Residues of $K_2$]\label{Thm:Main2}
    For $x=y=\pm1$, the simple pole at $t=2$ has residue $e^{-2\gamma_E}$ while the other possible simple poles at $t=-2m$, $m\in\mathbb{N}$ have residues given by (see Def. \ref{Def:pj}) $e^{2m\gamma_E}p_{2+2m,S_2}(-2m)$
    where the sequence $S_2$ satisfies $s_k^{(2)}=-B_k/k,$ $k\in\mathbb{N}$.

For $x=y$, $x\neq\pm1$, letting $x=\cos(\alpha),\ \alpha\in(0,\pi)$, the simple pole at $t=1$ has residue $e^{-\gamma_E}\pi^{-1}\csc(\alpha)$,
whereas the residues at the other possible simple poles $t=-m$, $m\in\mathbb{N}$, are given by
\begin{align}\label{Eq:GeneralResidues}
\frac{2e^{\gamma_E m}}{\pi}\bigg[&\sum_{\ell=1}^{m+1}(-1)^{\ell}\ell! p_{m+1-\ell,S_2}(-m)\sum_{k=0}^{\lfloor\ell/2\rfloor} \frac{(1/2)_k^2}{(k!)^2}b_{\ell-2k,k}(\alpha)\\
&+\frac{1}{2}\sum_{\ell=0}^{m}(-1)^{\ell}\ell! p_{m-\ell,S_2}(-m)\sum_{k=0}^{\lfloor\ell/2\rfloor} \frac{(1/2)_k^2}{(k!)^2}b_{\ell-2k,k}(\alpha)\bigg],\notag
\end{align}
where $(a)_k$ denotes Pochhammer's symbol, $S_2$ is the same as before, and
\begin{align}
&b_{i,k}(\alpha)=\sum_{j=0}^{\lfloor i/2\rfloor} d_{k+j,k}(\alpha)c_{i-2j}(\alpha),\quad i,k\in\mathbb{N}_0,\quad c_0(\alpha)=\frac{1}{2\sin(\alpha)},\notag\\
&c_j(\alpha)=\begin{cases}
\sin(\alpha)\big(-[i\cot(2\alpha)]^{\frac{1+(-1)^j}{2}}\frac{1}{j!}\Li_{-j}\big(e^{2i\alpha}\big)(\alpha)-\sum_{n=1}^{j-1}c_n(\alpha) c_{j-n}(\alpha)\big),& \alpha\in(0,\pi)\backslash\{\tfrac{\pi}{2}\},\\[2mm]
\frac{(2^{j+1}-1)B_{j+1}}{(j+1)!},& \alpha=\tfrac{\pi}{2},
\end{cases}\\
&d_{j,0}(\alpha)=\delta_{j,0},\quad d_{j,k}(\alpha)=\frac{1}{(2j)!}\sum_{\ell=k}^{j}\binom{\ell-1}{k-1}\frac{(-1)^{\ell-k}}{4^\ell \sin^{2\ell}(\alpha)}
\sum_{m=0}^{2\ell}
(-1)^m
\binom{2\ell}{m}
(\ell-m)^{2j},\ j\geq k\geq 1,\notag
\end{align}
with $\Li_s(z)$ denoting the polylogarithm function -- see Remark \ref{Rem:sqrt} and Appendix \ref{AppPoly} for alternate representations of $c_k(\alpha)$ and $d_{j,k}(\alpha)$, respectively.
\end{thm}

While the formulas above appear somewhat complicated, they can readily be calculated for small values of $m$ by means of a simple computer algebra program. Doing so, one is led to the following conjectures, which we pose as open problems. In particular, we conjecture that $K_2(t,\pm 1, \pm 1)$ has simple poles precisely at $t=2$ and $t=-2m$, $m\in\mathbb{N}$, and that in the more general case when $x=y\neq\pm1$, $K_2(t,x,x)$ has simple poles precisely at $t=-(2j-1)$, $j\in\mathbb{N}_0$ with corresponding residues given by the first line in \eqref{Eq:GeneralResidues}.

\begin{prob}\label{Problem2}
Consider the notation of Theorem \ref{Thm:Main2}. Then the following are conjectured to be true:\\[1mm]
$(i)$ $p_{2+2m,S_2}(-2m)\neq0$ for all $m\in\mathbb{N}.$\\[1mm]
$(ii)$ The second line in \eqref{Eq:GeneralResidues} is identically zero (resp., nonzero) for odd (resp., even) $m\geq0$ for all $\alpha\in(0,\pi)$.\\[1mm]
$(iii)$ The first line in \eqref{Eq:GeneralResidues} is nonzero for every $\alpha\in(0,\pi)$ whenever $m\in\mathbb{N}_0$.\\[1mm]
$(iv)$ For every $\alpha\in(0,\pi)$, \eqref{Eq:GeneralResidues} is equal to zero whenever $m\geq0$ is even.
\end{prob}

\begin{remark}
Notice that $(ii)$--$(iv)$ are natural generalizations of $(i)$ and the identities appearing in Corollary \ref{cor1} and \eqref{Eq:75} arising in our analysis. The importance of these statements lies in the fact that combining them allows one to prove exactly when the residues are zero/nonzero, hence showing the exact pole structure of $K_2(t,x,x)$. For instance, $(i)$ would show that the actual simple poles when $x=y=\pm1$ are precisely at $t=2$ and $t=-2m$, $m\in\mathbb{N}$.
Likewise, $(ii)$--$(iv)$ would show that the more general case $x=y\neq\pm1$ has simple poles precisely at $t=-(2j-1)$, $j\in\mathbb{N}_0$.

Finally, we point out that Remark \ref{Remark2} shows that Open Problems \ref{Problem1} and \ref{Problem2} $(i)$ are in fact two sides of the same coin and can be combined as stating $p_{j+2,S_2}(-j)\neq0$ for all $j\in\mathbb{N}$.
\end{remark}

\subsection{On some Bernoulli identities}\label{S:Bernoulli} Here we highlight some curious identities involving Bernoulli polynomials which are proven in Section \ref{Sect:poly} as a byproduct of our explicit calculations of the residues of $K_1(t,\varphi,\varphi)$ and $K_2(t,x,x)$. Many of these identities appear to be new, resembling more intricate analogues of `convolutional' identities for Bernoulli numbers -- see \cite{DS13,G05,Zag14}.

Let $S_2(s):=\{s_{k}^{(2)}(s)\}_{k \in \N}$ with $s_{k}^{(2)}(s):=-B_k(s)/k$, and where $B_k(s)$ is the $k$th Bernoulli polynomial. In Proposition \ref{lem6}, we show that for $j \in \N$, the quantity $p_{j,S_2(s)}(1-j)$ (cf. Def. \ref{Def:pj}) is independent of $s$, a fact which we use later to show that $p_{j,S_2(1)}(1-j)=p_{j,S_2}(1-j)=0$ for $j$ odd. Using \eqref{Eq:matrix}, this implies that the determinant of the following matrix
\begin{align}\label{Eq:BerMatrix}
    \mathcal{N}_{j,S_2(s)}(1-j) = \begin{pmatrix}
        (j-1)B_1(s) && 1 && 0  && \cdots &&\cdots\\ (j-1)B_2(s) && (j-1)B_1(s) &&2&&0 &&\cdots \\ 
        (j-1)B_3(s) && (j-1)B_2(s)&& (j-1)B_1(s)&& 3 &&\cdots \\ \vdots && &&  && \ddots \\ (j-1)B_j(s) && (j-1)B_{j-1}(s) && \cdots &&  \cdots && (j-1)B_1(s)
    \end{pmatrix},\ j\in\mathbb{N}.
\end{align}
is independent of $s$ and identically zero for $j$ odd. For specific choices of $j$ even, this constant can be readily calculated, leading for instance to the following identities, which involve linear combinations of products of Bernoulli polynomials whose indices sum to $j$:\footnote{This corrects the second to last identity claimed in \cite[Sect. 5]{E13}.} 
\begin{gather}
    \frac{B_1(s)^2}{2}-\frac{B_2(s)}{2} = \frac{1}{24}, \quad  -\frac{4}{3}B_1(s)^3+2B_1(s)B_2(s)-\frac{2}{3}B_3(s)=0,\notag\\ \frac{27}{8}B_1(s)^4 -\frac{27}{4}B_1(s)^2B_2(s)+3B_1(s)B_3(s)+\frac{9}{8}B_2(s)^2-\frac{3}{4}B_4(s)=-\frac{9}{640}, \\ -\frac{128}{15}B_1(s)^5\hspace{-.05cm}+\frac{64}{3}B_1(s)^3B_2(s)\hspace{-.05cm}-\hspace{-.05cm}8B_1(s)B_2(s)^2\hspace{-.05cm}-\hspace{-.05cm}\frac{32}{3}B_1(s)^2B_3(s)\hspace{-.05cm}+\hspace{-.05cm}4B_1(s)B_4(s)\hspace{-.05cm}+\hspace{-.05cm}\frac{8}{3}B_2(s)B_3(s)\hspace{-.05cm}-\hspace{-.05cm}\frac{4}{5}B_5(s)\hspace{-.05cm}=0.\notag
\end{gather}
More generally, for some real constants $C_j$ independent of $s$, these identities can be summarized as
\begin{equation}\label{Eq:18}
\sum_{k=0}^j (1-j)^k \sum \frac{1}{\ell_1!\ell_2!\dots \ell_{j-k+1}!}\bigg(\frac{B_1(s)}{1}\bigg)^{\ell_1} \bigg(\frac{B_2(s)}{2}\bigg)^{\ell_2} \dots \bigg(\frac{B_{j-k+1}(s)}{j-k+1}\bigg)^{\ell_{j-k+1}}=\begin{cases}
C_j,& j \text{ even},\\
0,& j \text{ odd},
\end{cases}
\end{equation}
with the inner sum being over all sequences $\ell_1,\ell_2,\dots,\ell_{j-k+1}$ of nonnegative integers such that
\begin{align}
\ell_1+\ell_2+\dots+\ell_{j-k+1}=k,\quad \ell_1+2\ell_2+\dots+(j-k+1)\ell_{j-k+1}=j.
\end{align}
Open Problem \ref{Problem1} is equivalent to showing $C_j\neq0$ for all even $j$ by Remark \ref{Remark2}, a fact we conjecture to be true but have not been able to prove. Finding a closed form expression for these values is therefore of interest.

We also note that Proposition \ref{lem7} leads to another sequence of identities involving Bernoulli numbers. For instance, setting $m =4$ and $q= 3$ in \eqref{eq:36} yields
\begin{align}
    2B_1^3 -3B_1^2-6B_1B_2+3B_2+4B_3=0,
\end{align}
which in tandem with the cubic identity above yields $B_1^2 +B_1B_2-B_2-B_3 =0$. A more general identity which we utilize is given in \eqref{Eq:75}: $p_{j,S_2(1)}(2-j)=-\frac{1}{2} p_{j-1,S_2(1)}(2-j)$ for odd $j\geq1$. This is a quite curious identity whenever considered in terms of Bernoulli numbers as in \eqref{Eq:18} or via the matrix formulation \eqref{Eq:BerMatrix}.

See Appendix \ref{AppPoly} for similar identities involving Bell polynomials and polylogarithms stemming from our general analysis, with Bernoulli identities as special cases.

\subsection{Notation}\label{S:Notation}
For simplicity and readability, we collect all the functions and constants explicitly here for reference that will be used freely throughout:
\begin{itemize}
    \item $\gamma_E$ denotes the Euler-Mascheroni constant;
    \item $B_k$ denotes the Bernoulli numbers with the convention $B_1=-1/2$;
    \item $B_k(s)$ denotes the Bernoulli polynomials which satisfy $B_k(0)=B_k$;
    \item $H_n:=\sum_{k=1}^n \frac{1}{k}$ denotes the $n$th harmonic number;
    \item $(a)_k$ denotes Pochhammer's symbol;
    \item $\zeta_R(s)=\sum_{n\in\mathbb{N}}n^{-s},$ $\Re(s)>1$, denotes the Riemann zeta function;
    \item $\eta(s)=\sum_{n\in\mathbb{N}}(-1)^{n-1}n^{-s}$, $\Re(s)>0$, denotes the Dirichlet eta function (or alternating zeta function);
    \item $\Li_s(z) :=\sum_{n\in \N}z^n n^{-s}$ denotes the polylogarithm function (for appropriate $z$ and $s$);
    \item $\psi(s) :=\frac{\Gamma'(s)}{\Gamma(s)}$ denotes the digamma function;
    \item $\widehat{B}_{n,k}$ denotes the \textup{partial ordinary Bell polynomials} \cite[p. 136]{Co74} \begin{equation}
\widehat{B}_{n,k}(z_1,z_2,\dots,z_{n-k+1})=\sum \frac{k!}{j_1!j_2!\dots j_{n-k+1}!}z_1^{j_1} z_2^{j_2} \dots z_{n-k+1}^{j_{n-k+1}},
\end{equation}
with the sum being over all sequences $j_1,j_2,\dots,j_{n-k+1}$ of nonnegative integers such that
\begin{align}
j_1+j_2+\dots+j_{n-k+1}=k,\quad j_1+2j_2+\dots+(n-k+1)j_{n-k+1}=n.
\end{align}
\end{itemize}

\section{Proofs of main results}\label{S2}

In order to prove our main results, we begin with an auxiliary lemma which essentially isolates the role of each term in the large $n$ asymptotics of the eigenvalues when analytically continuing the spectral series representations.

\begin{lem}\label{Lem:GenAsymp}
Let $\mu_n$ have as its asymptotic expansion $s_{-1}\ln(n)+\sum_{k\in\mathbb{N}_0}s_{k}n^{-k}$ as $n\to\infty$ for a real sequence $S=\{s_n\}_{n=-1}^\infty$. Then, for each $M\in\mathbb{N}$, on each compact set $K$ of the complex $t$-plane one has
\begin{align}\label{Eq:6a}
    e^{-t\mu_n} = n^{-s_{-1}t}e^{-s_0 t}\sum_{j=0}^{M-1}\frac{p_{j,S}(t)}{n^j}+e^{-s_0 t}O(n^{-s_{-1}\Re(t)-M}),
\end{align}
where $p_{j,S}(t)$ are polynomials in $t$ of degree $j$ which depend on $S$ given as in Definition \ref{Def:pj} while the implicit constant on the remainder term only depends on $M$ and the compact set $K$.
\end{lem}
\begin{proof}
We being by writing for each $M\in\mathbb{N}$,
\begin{align}\label{Eq:7a}
e^{-t\mu_n}=n^{-s_{-1}t}e^{-s_0 t}e^{-t D_{M,S}(n)}e^{-t R_{M,S}(n)},\ \text{ where }\ D_{M,S}(n)=\sum_{k=1}^{M-1} s_kn^{-k},
\end{align}
and $R_{M,S}(n)$ represents the remainder of the expansion for $\mu_n$.

Now for any compact set $K$ in the $t$-plane, we estimate as follows: First, for large enough $n$ one has, since $R_{M,S}(n)=O(n^{-M})$,
\begin{equation}
\big|e^{-t R_{M,S}(n)}\big|\leq 1+\big|e^{-t R_{M,S}(n)}-1\big|\leq 1+|t| |R_{M,S}(n)|e^{|t||R_{M,S}(n)|}\leq 1+ C_{M,K} n^{-{M}},
\end{equation}
for some constant $C_{M,K}$ which will only depend on the fixed $M\in\mathbb{N}$ and the chosen compact set $K$.

On the other hand, one also has that
\begin{equation}\label{Eq:9a}
e^{-t D_{M,S}(n)}=\sum_{j=0}^{M-1}\frac{(-t D_{M,S}(n))^j}{j!}+E_{M,S}(t,n)=\sum_{j=0}^{M-1}\frac{p_{j,S}(t)}{n^j}+O(n^{-M}),
\end{equation}
since for large enough $n$ one has by Taylor's Remainder Theorem, $|E_{M,S}(t,n)|\leq \widetilde{C}_{M,K}n^{-M}$ for some constant $\widetilde{C}_{M,K}$ which once again only depends on $M$ and the set $K$. Here $p_{j,S}(t)$ is the polynomial in $t$ found via expanding $(-t D_{M,S}(n))^j$ to isolate $n^{-j}$ for $j=0,\dots,M-1$.

Combining these estimates now yields \eqref{Eq:6a} via \eqref{Eq:7a}.

To prove \eqref{Eq:12a}, recall that these polynomials came from expanding the exponential $e^{-t D_{M,S}(n)}$. Let us momentarily consider the exponential $e^{-t D_S(n)}$ instead where $D_S(n)=\lim_{M\to\infty} D_{M,S}(n)$, that is, it includes the entire asymptotic series for $\mu_n-s_{-1}\ln(n) -s_0$ so that we consider the formal expression
\begin{equation}\label{Eq:17a}
\exp\bigg(-t\sum_{k\in\mathbb{N}}s_kn^{-k}\bigg).
\end{equation}
We now compare this with the generating function for the ordinary Bell polynomials given by
\begin{equation}\label{Eq:BellGena}
\exp\bigg( u\sum_{\ell\in\mathbb{N}} x_\ell y^\ell\bigg)=\sum_{j\geq k\geq 0}^\infty \widehat{B}_{j,k}(x_1,\dots,x_{j-k+1})\frac{u^k}{k!}y^j.
\end{equation}
In particular, we let $u=-t$, $y=n^{-1}$, and identify the sequences with one another to arrive at \eqref{Eq:12a}.
\end{proof}

\subsection{On the polynomials \texorpdfstring{$p_{j,S}(t)$}{p}}\label{Sect:poly}

We next turn to a more detailed analysis of the polynomials $p_{j,S}(t)$ as needed for our later calculation of residues, beginning with a general observation.

\begin{lem}\label{Lem:Pjident}
Let $p_{j,S}(t)$ be given as in \eqref{Eq:12a} and assume the sequence $S$ satisfies $s_{2k-1}\equiv0$ for all $k\in\mathbb{N}$. Then $p_{j,S}(t)\equiv0$ for all odd $j\geq1$.
\end{lem}
\begin{proof}
We offer two proofs of this result: one using combinatorial properties of Bell polynomials and the other using the matrix formulation \eqref{Eq:matrix}.

For the former, note the ordinary partial Bell polynomials $\widehat{B}_{j,k}$ are identically zero whenever $j\geq1$ is odd and all odd entries of the sequence defining them are 0. One can see this from \cite[Eq. (2.4)]{OS22}:
\begin{align}\label{Eq:BellEq}
    \widehat{B}_{j,k}(s_1,...,s_{j-k+1}) = \sum_{j_1+\cdots+j_k=j}s_{j_1}\cdots s_{j_k},
\end{align}
which represents the ordinary partial Bell polynomial as a sum over partitions of $j$ of size $k$, where the summand corresponding to the partition $(j_1,...,j_k)$ is the product of the entries of $S$ with indices $j_1,...,j_k$. Since $j$ is odd, each partition $(j_1,...,j_k)$ must contain at least one odd integer $j_\ell$, and therefore the corresponding summand is identically zero.

Alternatively, from the matrix formulation \eqref{Eq:matrix} we see that the $(j+1)/2$ odd numbered rows have nonzero entries in only the $(j-1)/2$ even numbered columns. Therefore, the rows lie in a $(j-1)/2$ dimensional subspace, but there is one more row than this dimension, so they must be linearly dependent. Thus the full set of rows are linearly dependent, proving the determinant, and hence $p_{j,S}(t)$, is zero.
\end{proof}

Next, we specialize to the sequence of numbers appearing in the analysis of our second operator, $A_2$, via the harmonic numbers, $H_n$. Let the sequence $S_2=\{s_n^{(2)}\}_{n=-1}^\infty$ satisfy
\begin{align}\label{Eq:S2}
 s_{-1}^{(2)}=1,\ s_0^{(2)}=\gamma_E,\ s_k^{(2)}=-B_k/k,\quad k\in\mathbb{N}.
\end{align} 
For the consideration of the polynomials $p_{j,S_2}(t)$, it suffices to only consider the terms of the sequence with $k\neq0$. As such, noting the well-known relationship between harmonic numbers and the digamma function, $\psi(n+1)=H_n-\gamma_E,$ $n\in\mathbb{N}$, we see that \begin{align}\label{Eq:68}
    e^{-t\psi(n+1)} = n^{-t}\sum_{j=0}^{M-1}\frac{p_{j,S_2}(t)}{n^j}+O(n^{-\Re(t)-M}),\quad n\in\mathbb{N},
\end{align}
relating the analysis of the polynomials $p_{j,S_2}(t)$ with the study of exponentials of digamma functions and approximating the Euler--Mascheroni constant -- see, for example, \cite{CEV13,Ya12}.
In particular, this allows us to identify $p_{j,S_2}(t)$ with $G_j(-t,1)$ in equation (2.5) of the preprint \cite{E13}. As \cite{E13} does not seem to have been published and the notation differs significantly from the present paper, we include direct proofs of the relevant parts for our analysis here for the convenience of the reader. We also extend these results in additional directions relevant to our work. 

We start by writing (found by differentiating \cite[Eq. 5.11.8]{DLMF}, or alternatively \cite[Eq. 2.11.9]{Luke69})
\begin{align}
    \psi(n+s) \sim \ln n +\sum_{k\geq 1}\frac{(-1)^k B_{k}(s)}{k}n^{-k},\quad n\in\mathbb{N},
\end{align}
where $B_{k}(s)$ are the Bernoulli polynomials. Noting that $B_n(1)=(-1)^nB_n$, this will give a natural generalization of our identities. So we can write 
\begin{align}
    e^{-t\psi(n+s)} \sim n^{-t}\sum_{j=0}^{\infty} \frac{p_{j,S_2(s)}(t)}{n^j},
\end{align}
with $S_2(s):=\{s_k^{(2)}(s)\}_{k \in \N}$ with $s_k^{(2)}(s):=-B_k(s)/k$ (where the terms for $k=-1,0$ can safely be ignored) and we point out that $S_2=S_2(1)$. We begin with the following two lemmas:

\begin{prop}\label{lem5}
The following holds: $p_{j,S_2(1)}(t)=(-1)^j p_{j,S_2(0)}(t)$.
\end{prop}
\begin{proof}
Note first that $B_n(0)=B_n$ and $B_n(1)=(-1)^n B_n$, so that the sequences $S_2(0)$ and $S_2(1)$ only differ by the sign on their $k=1$ terms. Therefore, the result follows from \eqref{Eq:12a} and \cite[Eq. (2.19)]{OS22} since the remaining odd terms of the sequence are zero so that one has by \cite[Eq. (2.19)]{OS22},
\begin{equation}
\widehat{B}_{j,k}(s_1^{(2)}(0),\dots,s_{j-k+1}^{(2)}(0))=(-1)^j \widehat{B}_{j,k}(-s_1^{(2)}(1),\dots,s_{j-k+1}^{(2)}(1)),
\end{equation}
that is, $p_{j,S_2(1)}(t)=(-1)^j p_{j,S_2(0)}(t)$ by comparing to \eqref{Eq:12a}.
\end{proof}

\begin{prop}\label{lem6}
In the variable $s$, $p_{j,S_2(s)}(t)$ is of degree $j$ if $t\notin\{1-j,2-j,\dots,0\}$.
Furthermore, if $t=\ell-j$ for some $\ell\in\{1,\dots,j\}$, then $p_{j,S_2(s)}(\ell-j)$ is a polynomial in $s$ with degree $\ell-1$.

In particular, if $t=1-j$, then $p_{j,S_2(s)}(1-j)$ is independent of $s$ defining the sequence $S_2(s)$.
\end{prop}
\begin{proof}
First note that differentiating the  asymptotic expansion in $n$ or $s$ yields the same result as this simply becomes the exponential of the asymptotic expansion of the trigamma function (cf. \cite[Sect. 5.15]{DLMF}).
This gives the equality of formal power series
\begin{align}
    \sum_{j=1}^\infty (1-j-t)\frac{p_{j-1,S_2(s)}(t)}{n^{j+t}} = \sum_{j=0}^\infty \frac{\frac{\partial}{\partial s}p_{j,S_2(s)}(t)}{n^{j+t}},
\end{align}
yielding for $j\geq1$, $\frac{\partial}{\partial s}p_{j,S_2(s)}(t)=(1-j-t)p_{j-1,S_2(s)}(t).$
Iterating yields
\begin{equation}
\frac{\partial^k}{\partial s^k}p_{j,S_2(s)}(t)=\bigg[\prod_{\ell=1}^{k}(\ell-j-t)\bigg]p_{j-k,S_2(s)}(t),\quad k\in\{1,\dots,j\}.
\end{equation}
As $p_{0,S_2(s)}(t)\equiv 1$, we see that $p_{j,S_2(s)}(t)$ is of degree $j$ in $s$ if $t\notin\{1-j,2-j,\dots,0\}$.
Furthermore, if $t=\ell-j$ for some $\ell\in\{1,\dots,j\}$, then $p_{j,S_2(s)}(\ell-j)$ is a polynomial in $s$ with degree $\ell-1$.
\end{proof}

Combining the previous two results, we arrive at the following result important to the current study:
\begin{corollary}\label{cor1}
Whenever $t=1-j$, with $j$ odd,  $p_{j,S_2(s)}(1-j)\equiv 0$.\footnote{Notice that for $s=1/2$, we can in fact appeal to Lemma \ref{Lem:Pjident} to see the restriction on $t$ values is not needed for this unique choice of $s$ due to all odd terms being identically zero in the corresponding sequence. This also offers an alternate proof of the corollary since, once we have that the choice $t=1-j$ yields independence of $s$, choosing $s=1/2$ and $j$ odd shows the result via Lemma \ref{Lem:Pjident}.}
\end{corollary}

We also have the following general result which includes Corollary \ref{cor1} via Proposition \ref{lem6}:
\begin{prop} \label{lem7}
    Let $m\in\mathbb{N}$. Then for each $q \in \{1,\dots,m\}$ one has
    \begin{align} \label{eq:36}
    \sum_{k=1}^{q-1} \frac{(-1)^{m-k}}{k!} \Bigg[\prod_{\ell=1}^k (\ell-q)\Bigg]p_{m-k,S_2(1)}(q-m) = \begin{cases}
        2p_{m,S_2(1)}(q-m), &m\textnormal{ odd}, \\ 0, & m \textnormal{ even}.
    \end{cases}
\end{align}
\end{prop}
\begin{proof}
    From Proposition \ref{lem6} we have that $p_{m,S_2(s)}(q-m) = \sum_{k=0}^{q-1}\tfrac{A_k}{k!}s^k$, and one computes (with the product understood as equal to $1$ for $k=0$)
\begin{align}
   A_k = \frac{\partial^k}{\partial s^k}p_{m,S_2(s)}(q-m)\Bigg|_{s=0} = \Bigg[\prod_{\ell=1}^k(\ell-q)\Bigg] p_{m-k,S_2(0)}(q-m)= (-1)^{m-k}\Bigg[\prod_{\ell=1}^k (\ell-q)\Bigg] p_{m-k,S_2(1)}(q-m),
\end{align}
where we have used Proposition \ref{lem5} in the last equality. This proves the result upon setting $s = 1$ and considering the parity of $m$ in $p_{m,S_2(s)}(q-m) = \sum_{k=0}^{q-1}\tfrac{A_k}{k!}s^k$.
\end{proof}

Of most importance for our work will be the case of odd $m\geq1$ with $q=2$:
\begin{equation}\label{Eq:75}
p_{m,S_2(1)}(2-m)=-\frac{1}{2} p_{m-1,S_2(1)}(2-m).
\end{equation}

\begin{remark}\label{Remark2}
Note that $p_{j,q S}(t)=p_{j,S}(q t)$ by definition and \cite[Eq. (2.18)]{OS22}, so that we can include the analysis of the sequence associated with our first operator, $A_1$, via the identification $S_1=2 S_2(1/2)$, ignoring the $k=-1,0$ terms. This amounts to considering $p_{j,S_1}(T)=p_{j,S_2(1/2)}(2T)$, i.e., $s=1/2$ and $t=2T$, in the analysis here. Setting $T=\frac{1}{2}-m$ and $j=2m$ yields that Open Problem \ref{Problem1} is equivalent to showing $p_{j,S_2(s)}(1-j)\neq 0$ for even $j\geq 2$ for some, hence every, $s$, via Proposition \ref{lem6}.

Moreover, by \eqref{Eq:75}, this is equivalent to showing that $p_{j+1,S_2(1)}(1-j)\neq0$ for even $j\geq 2$, whereas Open Problem \ref{Problem2} $(i)$ is equivalent to showing that $p_{j+2,S_2(1)}(-j)\neq0$ for  even $j\geq 2.$ Therefore, Open Problems \ref{Problem1} and \ref{Problem2} $(i)$ can be surprisingly combined into one conjecture: $p_{j+2,S_2(1)}(-j)\neq0$ for all $j\in\mathbb{N}$.
\end{remark}

\subsection{The easy cases}

We now turn to the explicit analytic continuation of $K_1(t,\varphi,\theta)$ followed by special cases of $K_2(t,x,y)$ which follow similarly. For specific notation, we once again refer to Section \ref{S:Notation}.

The following result includes Theorem \ref{Thm:Main3} upon verifying the residue of $K_1(t,\varphi,\varphi)$ for $t=\tfrac{1}{2}$ is nonzero.

\begin{prop}
Let the sequence $S_1=\{s_n^{(1)}\}_{n=-1}^\infty$ satisfy
\begin{align}
 s_{-1}^{(1)}=2,\ s_0^{(1)}=4\ln(2)+2\gamma_E,\ s_k^{(1)}=(1-2^{1-k})\frac{2 B_{k}}{k}=-\frac{2}{k}B_k(1/2),\quad k\in\mathbb{N}.
\end{align} 
Then for each $M\in\mathbb{N}$ and $\Re(t) > \frac{1-M}{2}$, one has for $\varphi\in[0,2\pi)$,
\begin{align}\label{Eq:29a}
    K_1(t,\varphi,\varphi) = \frac{1}{2\pi}+\frac{e^{-[4\ln(2)+2\gamma_E] t}}{\pi}\sum_{j=0}^{M-1}p_{j,S_1}(t)\zeta_R(j+2t)+F_{M,1}(t,\varphi),
\end{align}
while for $\varphi,\theta\in[0,2\pi)$, $\varphi\neq\theta$,
\begin{align}\label{Eq:29aa}
    K_1(t,\varphi,\theta) = \frac{1}{2\pi}+\frac{e^{-[4\ln(2)+2\gamma_E] t}}{2\pi}\sum_{j=0}^{M-1}p_{j,S_1}(t)\big[\Li_{2t+j}(e^{i(\varphi-\theta)})+\Li_{2t+j}(e^{-i(\varphi-\theta)})\big]+ F_{M,2}(t,\varphi,\theta),
\end{align}
where $p_{j,S_1}(t)$ are as given in \eqref{Eq:12a} with $S_1$ as above, and, for each fixed $\varphi,\theta$ given as above, $F_{M,1}(t,\varphi)$ and $F_{M,2}(t,\varphi,\theta)$ are analytic for $\Re(t) > \frac{1-M}{2}$.

In particular, $K_1(t,\varphi,\theta)$ can be analytically continued to an entire function of $t$ for each $\varphi,\theta\in[0,2\pi)$ such that $\varphi\neq\theta$ whereas for each $\varphi\in[0,2\pi)$, $K_1(t,\varphi,\varphi)$ can be continued to a meromorphic function with possible simple poles at $t=\frac{1}{2}-m$, $m\in\mathbb{N}_0$ with corresponding residues given by $\frac{e^{-2\ln(2)-\gamma_E}}{2\pi}e^{[4\ln(2)+2\gamma_E] m}p_{2m,S_1}(\frac{1}{2}-m)$.
\end{prop}

\begin{proof}
We begin by recalling (cf. Lemma 12.2.2 in \cite{MNP00}) that the eigenvalues for this problem are given by $2\lambda_n = 4\sum_{j=0}^{n-1}(2j+1)^{-1} = 2(2H_{2n}-H_n)$. Therefore we have the asymptotic expansion 
\begin{align}
    2\lambda_n \sim 2\ln n +4\ln(2)+2\gamma_E+\sum_{k \geq 1}\frac{B_{2k}}{k}(1-2^{1-2k})n^{-2k} = s_{-1}^{(1)}\ln(n)+\sum_{k\in \N_0}s_k^{(1)} n^{-k},
\end{align}
agreeing with the definition of $S_1$ above.
Moreover, each $\lambda_n$ has corresponding normalized eigenfunctions $\frac{1}{\sqrt{2\pi}}e^{\pm i n\varphi}$. Therefore, using Lemma \ref{Lem:GenAsymp} and summing over $n \in\mathbb{Z}\backslash \{0\}$ (noting this becomes two times the sum over $n\in\mathbb{N}$) yields the representation \eqref{Eq:29a} for $K_1(t,\varphi,\varphi)$ with $\varphi \in [0,2\pi)$ except for the fact that the sum over the remainder terms, $F_{M,1}(t,\varphi)$, is analytic in $t$ for $\Re(t)>(1-M)/2$. However, this follows from the typical arguments since we know for $\Re(t)\geq\sigma>(1-M)/2$ the remainder terms can be dominated by the series with terms $C n^{-2\sigma-M+1}$. Thus the sum over the remainder terms converges absolutely and locally uniformly on the half-plane $\Re(t)>(1-M)/2$, yielding that $F_{M,1}(t,\varphi)$ is analytic on this half-plane.
The sum of Riemann zeta functions now yields possible simple poles at $t = \frac{1-\ell}{2}$ for $\ell \in \N_0$ with corresponding residues $\frac{e^{-s_0^{(1)}/2}}{2\pi}e^{s_0^{(1)}\ell/2}p_{\ell,S_1}(\tfrac{1-\ell}{2})$ (as the Riemann $\zeta$-function contributes $1/2$ to the residues). But since $s_{2k-1}^{(1)}=0$ for all $k \in \N$, by Lemma \ref{Lem:Pjident} we have that $p_{\ell,S_1}\equiv 0$ for $\ell$ odd, so that the apparent simple poles at $\frac{1-\ell}{2}$ for $\ell$ odd are in fact absent, yielding the result given above.

Similarly, if $\varphi,\theta\in[0,2\pi)$, $\varphi\neq\theta$, we have the representation
\begin{align}
    K_1(t,\varphi,\theta) = \frac{1}{2\pi}+\frac{e^{-[4\ln(2)+2\gamma_E] t}}{2\pi}\sum_{j=0}^{M-1}p_{j,S_1}(t)\sum_{n \in \Z\setminus\{0\}}e^{in(\varphi-\theta)}|n|^{-2t-j} + F_{M,2}(t,\varphi,\theta),
\end{align}
where for any $\varphi,\theta$ given as above, $F_{M,2}(t,\varphi,\theta)$ is also analytic for $\Re(t) > \frac{1-M}{2}$. Setting $\varphi-\theta = \nu$, this reduces to studying
\begin{align}
    \sum_{n\in \Z \setminus \{0\}} e^{in\nu}|n|^{-2t-j} = \sum_{n \in \N}(e^{in\nu }+e^{-in\nu})\,n^{-2t-j} = \Li_{2t+j}(e^{i\nu})+\Li_{2t+j}(e^{-i\nu}),
\end{align}
where $\Li_s(z) :=\sum_{n\in \N}\frac{z^n}{n^s}$ is the usual polylogarithm function. From the Hankel contour representation of the polylogarithm, one obtains the series for $\Im(\mu) \in (0,2\pi)$ \cite[Eq. 25.12.12]{DLMF}
\begin{align}
    \Li_s(e^\mu) = \Gamma(1-s)(-\mu)^{s-1}+\sum_{k\in \N_0}\frac{\zeta_R(s-k)}{k!}\mu^k.
\end{align}
One now readily checks that the apparent poles of the $\Gamma$-function and the Riemann zeta functions at $s=n\in\mathbb{N}$ in fact cancel so that the polylogarithm is entire. Therefore, as a composition of entire functions of $t$, we have that $\Li_{2t+j}(e^{\pm i\nu})$ is entire in $t$ for all $j \in \N_0$, and the result follows.
\end{proof}

Next we consider two edge cases of $K_2(t,x,y)$ which yield explicit formulas in terms of well-known zeta functions similarly to the previous result for $K_1$.

\begin{lem}\label{Lem:Spec}
Let the sequence $S_2=\{s_n^{(2)}\}_{n=-1}^\infty$ satisfy
\begin{align}
 s_{-1}^{(2)}=1,\ s_0^{(2)}=\gamma_E,\ s_k^{(2)}=-B_k/k,\quad k\in\mathbb{N}.
\end{align} 
Then for each $M\in\mathbb{N}$, one has for $a\in\{-1,1\}$ and $\Re(t)>2-M$,
\begin{align}
&K_2(t,a,a)=\frac{1}{2}+e^{-t\gamma_E}\sum_{j=0}^{M-1}p_{j,S_2}(t)\bigg[\zeta_R(t+j-1)+\frac{1}{2}\zeta_R(t+j)\bigg]+G_{M,1}(t,a), \label{Eq:Mero}\\
&K_2(t,a,-a)=\frac{1}{2}-e^{-t\gamma_E}\sum_{j=0}^{M-1}p_{j,S_2}(t)\bigg[\eta(t+j-1)+\frac{1}{2}\eta(t+j)\bigg]+G_{M,2}(t,a),\label{Eq:Ent}
\end{align}
where $p_{j,S_2}(t)$ are as given in \eqref{Eq:12a} with $S_2$ as above and, for each $a\in\{-1,1\}$, $G_{M,m}(t,a)$, $m=1,2,$ are analytic for $\Re(t)>2-M$.

In particular, $K_2(t,a,-a)$ can be analytically continued to an entire function of $t$ whereas $K_2(t,a,a)$ can be continued to a meromorphic function which has a simple pole at $t=2$ with residue $e^{-2\gamma_E}$ while other possible simple poles occur at $t=-2m$, $m\in\mathbb{N}$ with residues $e^{2m\gamma_E}p_{2+2m,S_2}(-2m)$.
\end{lem}
\begin{proof}
First notice that the definition of $S_2$ agrees with the asymptotics of the harmonic numbers, $H_n \sim \ln(n) +\gamma_E+\frac{1}{2n}-\sum_{k\geq 1}\frac{B_{2k}}{2k}n^{-2k}$ as $n \rightarrow \infty$.

We now begin by analyzing $K_2(t,a,a)$ by summing over $n\in\mathbb{N}$ in \eqref{Eq:6a} after multiplying by $(2n+1)/2$ as $P^2_n(\pm1)=1$. This immediately yields \eqref{Eq:Mero} except for the fact that the sum over the remainder terms, $G_{M,1}(t,a)$, is analytic in $t$ for $\Re(t)>2-M$. However, this follows similarly to before since we know for $\Re(t)\geq\sigma>2-M$ the remainder terms can be dominated by the series with terms $C n^{-\sigma-M+1}$. Thus the sum over the remainder terms converges absolutely and locally uniformly on the half-plane $\Re(t)>2-M$, yielding that $G_{M,1}(t,a)$ is analytic on this half-plane. The statement regarding the structure of the analytic continuation now follows from these representations since they hold for each $M\in\mathbb{N}$ allowing one to move the half-plane $\Re(t)>2-M$ through the entire complex $t$-plane. In particular, as $\zeta_R(s)$ has a simple pole at $s=1$ with residue equal to $1$, one sees that there are possible simple poles at $t=2-m$, $m\in\mathbb{N}_0$ with corresponding residues given by $e^{-2\gamma_E}$ (as $p_{0,S_2}(t)=1$) for $m=0$ and, otherwise,
    \begin{equation}\label{Eq:50a}
    e^{-(2-m)\gamma_E}\bigg[p_{m,S_2}(2-m)+\frac{1}{2}p_{m-1,S_2}(2-m)\bigg],\quad m\in\mathbb{N}.
    \end{equation}
Notice that \eqref{Eq:75} implies that the residues for odd $m$ are zero, so none of these are poles. Moreover, $m=2$ results in a zero residue since $p_{j,S_2}(0)=\delta_{j,0}$, reducing the possible simple poles to $t=2$ and $t=2-m$ with $m\in\mathbb{N}$ even.
Finally, Corollary \ref{cor1} implies that the second term in \eqref{Eq:50a} is zero for $m\in\mathbb{N}$ even, finishing the analysis of $K_2(t,a,a)$.

The analysis of $K_2(t,a,-a)$ follows analogously after multiplying \eqref{Eq:6a} by $(2n+1)/2$ and $P_n(-1)=(-1)^n$ which yields negative one times appropriate Dirichlet eta functions once summing over $n\in\mathbb{N}$ rather than Riemann zeta functions. As $\eta(s)$ is entire, we conclude that $K_2(t,a,-a)$ is entire as well.
\end{proof}

\subsection{General setting}

We next include the Legendre polynomials directly into our analysis of $K_2$, which proves to be much more involved and interesting. We break the proof into two parts by first showing the analytic structure and then carefully calculating the residues.

\begin{proof}[Proof of Theorem \ref{Thm:Main} (Analytic structure).]
First notice that Lemma \ref{Lem:Spec} directly treats the cases whenever $x,y\in\{-1,1\}$.
Thus we can safely assume $x,y\in[-1,1],$ $(x,y)\neq(\pm1,\pm1),(\pm1,\mp1)$ throughout the proof.

To treat the general setting, note it suffices to utilize the sequence $S_2$ defined in \eqref{Eq:S2}, then multiply \eqref{Eq:6a} by $[(2n+1)/2]P_n(x)P_n(y)$ and consider only the terms contributing to the summation over $n$ for each fixed $M\in\mathbb{N}$. That is, we fix $M\in\mathbb{N}$ and let $j\in\{0,1,\dots,M-1\}$ to consider
\begin{align}\label{Eq:Hilbsconv}
\begin{split}
h_j(t,x,y)=\sum_{n\in \N}\frac{2n+1}{2}\frac{P_n(x)P_n(y)}{n^{t+j}},\quad \Re(t)>-j+1-(1/2)(1-\delta_{x,\pm1}+1-\delta_{y,\pm1})(1-\delta_{x,y}),&\\
x,y\in[-1,1],\ (x,y)\neq(\pm1,\pm1),(\pm1,\mp1),&
\end{split}
\end{align}
where $\delta_{i,j}$ denotes the Kronecker delta function and the region of convergence follows from the leading term and oscillations of Hilb's formula for the asymptotics of $P_n(x)$ (cf. Thm. 8.21.2 in \cite{Sz39}).

In our analysis below it will be convenient to break \eqref{Eq:Hilbsconv} into two sums via splitting the first fraction and simply consider (for the same range of $x,y$ values as before throughout the following)
\begin{equation}\label{Eq:tildeh}
\widetilde{h}_j(t,x,y)=\sum_{n\in \N}\frac{P_n(x)P_n(y)}{n^{t+j}},\quad \Re(t)>-j-(1/2)(1-\delta_{x,\pm1}+1-\delta_{y,\pm1})(1-\delta_{x,y}),
\end{equation}
noting that $h_j(t,x,y)=\widetilde{h}_{j-1}(t,x,y)+\tfrac{1}{2}\widetilde{h}_j(t,x,y)=\widetilde{h}_{j}(t-1,x,y)+\tfrac{1}{2}\widetilde{h}_j(t,x,y)$. We point out that this is a Dirichlet series with Legendre polynomial coefficients, though we will not use this fact in what follows.

Now recalling the Mellin transform
\begin{equation}
n^{-w}=\frac{1}{\Gamma(w)}\int_0^\infty s^{w-1}e^{-ns}\, ds,\quad \Re(w)>0,
\end{equation}
allows one to rewrite $\widetilde{h}_j(t,x,y)$ as (comparing $\Re(t+j)>0$ and the $t$ half-plane in \eqref{Eq:tildeh})
\begin{equation}
\widetilde{h}_j(t,x,y)=\frac{1}{\Gamma(t+j)}\int_0^\infty s^{t+j-1}\sum_{n\in \N} P_n(x)P_n(y) e^{-ns}\, ds,\quad \Re(t)>-j.
\end{equation}
Since $|e^{-ns}|<1$ for $n\in\mathbb{N}$, $s>0$, we can utilize the generating function derived in \cite{Max56}\footnote{We note that although the result in \cite{Max56} is correct, there is a typo in equation (12) of that paper, where the instance of $\cos^2\alpha$ appearing in the formulas for $A^2$ and $C^2$ should be replaced by $\sin^2\beta$.}:
\begin{equation}
\sum_{n\in\mathbb{N}_0}P_n(\cos\, \alpha) P_n(\cos\, \beta) z^n=\frac{\mathstrut_2 F_1(\tfrac{1}{2},\tfrac{1}{2};1;\tfrac{4z\sin(\alpha) \sin(\beta)}{1-2z\cos(\alpha+\beta)+z^2})}{\sqrt{1-2z\cos(\alpha+\beta)+z^2}},\quad |z|<1,\ \alpha,\beta\in[0,\pi],
\end{equation}
where $\mathstrut_2 F_1(a,b;c;q)$ is the hypergeometric function \cite[Sect. 15]{DLMF}. Hence, subtracting off the $n=0$ term and letting $z=e^{-ns}$ yields for $\alpha,\beta\in[0,\pi]$, $(\alpha,\beta)\neq(0,0),(0,\pi),(\pi,0),(\pi,\pi)$,
\begin{equation}\label{Eq:hint}
\widetilde{h}_j(t,\cos(\alpha),\cos(\beta))=\frac{1}{\Gamma(t+j)}\int_0^\infty s^{t+j-1}\bigg(\frac{\mathstrut_2 F_1(\tfrac{1}{2},\tfrac{1}{2};1;\tfrac{4e^{-s}\sin(\alpha) \sin(\beta)}{1-2e^{-s}\cos(\alpha+\beta)+e^{-2s}})}{\sqrt{1-2e^{-s}\cos(\alpha+\beta)+e^{-2s}}}-1\bigg)\, ds.
\end{equation}

Now as we are concerned with the possible singular structure with respect to $t$, we need to consider the large and small $s$ behavior of the integrand separately since $t$ only appears in the power of $s$ (and the reciprocal gamma function which is entire). We begin with the large $s$ behavior by first noting that since the argument of the hypergeometric function goes to 0 as $s\to\infty$ (and is identically $0$ if $\alpha$ and/or $\beta$ equals $0$ or $\pi$), the hypergeometric function behaves like 1 to leading order.
This implies that for large $s$ one has that the integrand is exponentially decaying since
\begin{equation}
s^{t+j-1}\bigg(\frac{1}{\sqrt{1-2a e^{-s}+e^{-2s}}}-1\bigg)=s^{t+j-1}\big(ae^{-s}+O(e^{-2s})\big).
\end{equation}
Therefore, away from $s=0$, the integral is analytic in $t$, so it suffices to consider the small $s$ behavior to probe any singularities with respect to $t$ of $\widetilde{h}_j$. 

We begin by noticing that since $\cos(\alpha+\beta)\neq1$ for any choices of $\alpha,\beta$ considered here, one has for the square root a series expansion of the form\footnote{Note that if the cases $\alpha=\beta=0,\pi$, that is, $x=y=\pm1$, are included in the analysis here, this expansion needs to include the term $1/s$. When integrating near small $s$ below, this additional singularity as $s\downarrow0$ is exactly the term that will cause the poles which were observed in Lemma \ref{Lem:Spec} to appear. As such, one can readily includes these cases here with this case distinction.}
\begin{equation}\label{Eq:squareroot}
\frac{1}{\sqrt{1-2e^{-s}\cos(\alpha+\beta)+e^{-2s}}}=\sum_{k\in\mathbb{N}_0}c_ks^k.
\end{equation}
Hence the problem has been reduced to finding the series expansion for the hypergeometric function appearing in the integrand.

We now focus on the argument of the hypergeometric function in the integral formula for $\widetilde{h}_j$ (where we exclude the cases $\alpha$ and/or $\beta$ equals $0$ or $\pi$ as the hypergeometric function considered here simply becomes identically equal to 1):
\begin{equation}\label{Eq:Halpha}
H_{\alpha,\beta}(s)=\frac{4e^{-s}\sin(\alpha) \sin(\beta)}{1-2e^{-s}\cos(\alpha+\beta)+e^{-2s}}=\frac{4e^{-s}\sin(\alpha) \sin(\beta)}{\big|1-e^{i(\alpha+\beta)}e^{-s}\big|^2},\quad \alpha,\beta\in(0,\pi).
\end{equation}
Note that for any fixed $\alpha,\beta\in(0,\pi)$, $H_{\alpha,\beta}(s)$ is a well-defined, decreasing, positive function of $s$ on $[0,1]$.
Since we are concerned with the $s\downarrow0$ behavior, first notice that we will be strictly inside the radius of convergence for the hypergeometric function provided $|H_{\alpha,\beta}(s)|=H_{\alpha,\beta}(s)<1$, When this occurs, we can simply expand by composing the power series of the hypergeometric function and that of $H_{\alpha,\beta}(s)$ to once again obtain a power series with terms $s^{k}$, $k\in\mathbb{N}_0$. The resulting series can then be integrated term-by-term.
Analyzing the restriction on $\alpha,\beta\in(0,\pi)$ for $H_{\alpha,\beta}(s)<1$ now reduces to considering $H_{\alpha,\beta}(0)$ by our previous decreasing observation. That is, we must exclude the case
\begin{equation}
2\sin(\alpha)\sin(\beta)= 1-\cos(\alpha+\beta) \Leftrightarrow \cos(\alpha-\beta)=1,
\end{equation}
that is, all $\alpha,\beta\in(0,\pi)$ can be treated by the above arguments except for the case $\alpha=\beta$. 
This implies that for $\alpha\neq\beta$ the small $s$ behavior of the integrand yields a series in $s^{t+j-1+k}$ with $k\in\mathbb{N}_0$, which, after integrating term-by-term, amounts to having simple poles exactly at $t+j=-k$ for the integral. For $\widetilde{h}_j(t,x,y)$, these apparent simple poles are exactly canceled by the simple zeros of the reciprocal gamma function in \eqref{Eq:hint}; hence, the resulting function is entire in $t$ for $\alpha,\beta\in[0,\pi]$, $\alpha\neq\beta$ (where the endpoints $0,\pi$ are included since, as we observed above, in those cases the hypergeometric function is identically equal to 1). Recalling that $h_j(t,x,y)=\widetilde{h}_{j-1}(t,x,y)+\tfrac{1}{2}\widetilde{h}_j(t,x,y)$, we can conclude that $h_j$, and therefore $K_2(t,x,y)$, is an entire function of $t$ for each fixed $x,y\in[-1,1]$ whenever $x\neq y$ (the edge cases were treated in Lemma \ref{Lem:Spec}).

Finally, we treat the diagonal case of $\alpha=\beta$, that is, the case $x=y$. In this case one finds that $H_{\alpha,\alpha}(0)=1$ and the hypergeometric function blows up as $s\downarrow0$. The exact behavior of this blow up will correspond to the type of singularity we encounter in $t$. Now since the first two parameters add up to the third in the hypergeometric function we consider here, one has the connection formula \cite[Eq. 15.8.10]{DLMF}
\begin{align}\label{Eq:2f1}
\mathstrut_2 F_1(\tfrac{1}{2},\tfrac{1}{2};1,q)&=-\frac{1}{\pi}\sum_{k\in\mathbb{N}_0} \frac{(1/2)_k^2}{(k!)^2}(1-q)^k[\ln(1-q)-2\psi(k+1)+2\psi((1/2)+k)],
\end{align}
where $(a)_k$ denotes Pochhammer's symbol and $\psi(z)$ the digamma function. As $H_{\alpha,\alpha}(s)\uparrow1$ as $s\downarrow0$, upon substituting the small $s$ series for $H_{\alpha,\alpha}$ and expanding, the right-hand side of \eqref{Eq:2f1} yields the small $s$ expansion of the hypergeometric function appearing in our generating function. In particular, expanding about $s=0$ now yields that this function behaves like $\ln(s)$ times nonnegative powers of $s$ plus a power series in $s^k$, $k\in\mathbb{N}$ (coming from expanding the logarithm). As we know the power series portion will not contribute any singularities to the problem by our previous arguments, we need only consider the new addition of a logarithmic term. This amounts to considering integrals of the form (the choice of upper limit of integration is immaterial as we will be concerned with whenever $t+j+k=0$ anyways)
\begin{align}\label{Eq:logpole}
\begin{split}
\int_0^1 s^{t+j-1+k}\ln(s)\, ds=\frac{s^{t+j+k}}{t+j+k}\ln(s)-\frac{s^{t+j+k}}{(t+j+k)^2}\bigg|_{s=0}^1=-\frac{1}{(t+j+k)^2},\quad k\in\mathbb{N}_0,\ \Re(t)>-j.
\end{split}
\end{align}
Notice that the logarithmic term causes a higher order pole to appear in the integral which can only be partially canceled by the simple zero of the reciprocal Gamma function. Therefore we are able to conclude that possible simple poles with respect to $t$ of $\widetilde{h}_j(t,x,x)$, $x\in(-1,1)$, appear precisely at $t=-\ell$, $\ell\in\mathbb{N}_0$ since $(j+k)\in\mathbb{N}_0$.
Recalling once again that $h_j(t,x,y)=\widetilde{h}_{j-1}(t,x,y)+\tfrac{1}{2}\widetilde{h}_j(t,x,y)$, we can conclude that the possible simple poles of $K_2(t,x,x)$, $x\in(-1,1)$ are at $t=1-\ell$, $\ell\in\mathbb{N}_0$, with the additional pole at $t=1$ coming from $\widetilde{h}_{j-1}(t,x,y)$.
\end{proof}

\subsubsection{On the careful calculation of residues}

In order to write more concise formulas for the residues in the general setting, we first delve into a more careful analysis of the functions given in \eqref{Eq:squareroot} and \eqref{Eq:Halpha} for $\alpha=\beta$.
Notice that \eqref{Eq:2f1} shows that we actually require an expression for the small $s$ series expansion of $(1-H_{\alpha,\alpha}(s))^k$ for each $k\in\mathbb{N}$ since this is the expression multiplying the logarithm:

\begin{lem}\label{Lem:Hkexp}
The following holds for $\alpha\in(0,\pi)$, $s\geq0$, and $k\in\mathbb{N}$:
\begin{equation}
(1-H_{\alpha,\alpha}(s))^k=\sum_{j=k}^\infty d_{j,k}(\alpha) s^{2j},\quad d_{j,k}(\alpha)=\frac{1}{(2j)!}\sum_{\ell=k}^{j}\binom{\ell-1}{k-1}\frac{(-1)^{\ell-k}}{4^\ell \sin^{2\ell}(\alpha)}
\sum_{m=0}^{2\ell}
(-1)^m
\binom{2\ell}{m}
(\ell-m)^{2j},\ j\geq k.
\end{equation}
\end{lem}
\begin{proof}
First note the following where we have used that $\cos(2\alpha)+2\sin^2(\alpha)=1$ and $1-e^{-s}=2e^{-s/2}\sinh(\tfrac{s}{2})$ (from which $\sinh^2(\tfrac{s}{2})=4^{-1}(e^{s/2}-e^{-s/2})^2$):
\begin{align}
1-H_{\alpha,\alpha}(s)&=\frac{(1-e^{-s})^2}{1-2e^{-s}\cos(2\alpha)+e^{-2s}}=\frac{4\sinh^2(\tfrac{s}{2})}{e^{s}-2\cos(2\alpha)+e^{-s}}=\frac{4\sinh^2(\tfrac{s}{2})}{e^{s}+e^{-s}-2+4\sin^2(\alpha)}\notag\\
&=\frac{\sinh^2(\tfrac{s}{2})}{4^{-1}(e^{s/2}-e^{-s/2})^2+\sin^2(\alpha)}=\frac{\sinh^2(\tfrac{s}{2})}{\sinh^2(\tfrac{s}{2})+\sin^2(\alpha)}.
\end{align}
Notice this is of the form $z/(z+a)=(z/a)[1+(z/a)]^{-1}$, so we can raise to the power $k$ and then use the negative binomial expansion to obtain (where we use $\binom{\ell-1}{\ell-k}=\binom{\ell-1}{k-1}$)
\begin{equation}
(1-H_{\alpha,\alpha}(s))^k=\sum_{m=0}^\infty (-1)^m\binom{k+m-1}{m}\frac{\sinh^{2(k+m)}(\tfrac{s}{2})}{\sin^{2(k+m)}(\alpha)}=\sum_{\ell=k}^\infty (-1)^{\ell-k}\binom{\ell-1}{k-1}\frac{\sinh^{2\ell}(\tfrac{s}{2})}{\sin^{2\ell}(\alpha)},\quad k\in\mathbb{N}.
\end{equation}
Substituting the following series expansion for $\sinh^{2\ell}(\tfrac{s}{2})$ into this series and interchanging sums (due to absolute convergence) finishes the proof:
\begin{equation}
\sinh^{2\ell}(\tfrac{s}{2})=\sum_{j=\ell}^\infty \frac{1}{4^\ell (2j)!}
\sum_{m=0}^{2\ell}
(-1)^m
\binom{2\ell}{m}
(\ell-m)^{2j}s^{2j}.
\end{equation}
\end{proof}

We now turn to the analysis of \eqref{Eq:squareroot} which is a bit more complicated:
\begin{lem}\label{Lem:Sqrt}
The following holds for $\alpha\in(0,\pi)$ and $s\geq0$: 
\begin{align}\label{Eq:Sqrtseries}
&Q_\alpha(s)=\frac{1}{\sqrt{1-2e^{-s}\cos(2\alpha)+e^{-2s}}}=\sum_{k\in\mathbb{N}_0}c_k(\alpha)s^k,\quad c_0(\alpha)=\frac{1}{2\sin(\alpha)},\\
&c_k(\alpha)=\begin{cases}
\sin(\alpha)\big(-[i\cot(2\alpha)]^{\frac{1+(-1)^k}{2}}\frac{1}{k!}\Li_{-k}\big(e^{2i\alpha}\big)-\sum_{n=1}^{k-1}c_n(\alpha) c_{k-n}(\alpha)\big),& \alpha\in(0,\pi)\backslash\{\tfrac{\pi}{2}\},\\[2mm]
\frac{(2^{k+1}-1)B_{k+1}}{(k+1)!},& \alpha=\tfrac{\pi}{2},
\end{cases}\ k\in\mathbb{N}.\notag
\end{align}
\end{lem}

\begin{remark}\label{Rem:sqrt}
Utilizing Bell polynomials results in the following alternate representation of the coefficients $c_k$:
\begin{equation}
c_k(\alpha)=\frac{1}{2\sin(\alpha)}\sum_{n=0}^k\frac{(-1)^n(2n)!}{16^{n}(n!)^2 \sin^{2n}(\alpha)} \widehat{B}_{k,n}(r_1,\dots,r_{k-n+1}),\quad r_\ell=(-1)^\ell\frac{2^\ell-2\cos(2\alpha)}{\ell!},\quad \ell\in\mathbb{N}.
\end{equation}
\end{remark}

\begin{proof}[Proof of Lemma \ref{Lem:Sqrt}.]
We begin with the case $\alpha=\tfrac{\pi}{2}$:
\begin{equation}
Q_{\tfrac{\pi}{2}}(s)=\frac{1}{1+e^{-s}}=\frac{1}{2}+\frac{1}{2}\tanh(\tfrac{s}{2})=\frac{1}{2}+\sum_{n\in\mathbb{N}}
\frac{(2^{2n}-1)B_{2n}}{(2n)!}\, s^{2n-1}.
\end{equation}

For the general case $\alpha\in(0,\pi)\backslash\{\tfrac{\pi}{2}\}$, we consider the square of $Q_\alpha(s)$ to write
\begin{align}
\label{eq:2.49}
\frac{2i\sin(2\alpha)}{1-2e^{-s}\cos(2\alpha)+e^{-2s}}&=\frac{e^{2i\alpha}}{1-e^{-s}e^{2i\alpha}}-\frac{e^{-2i\alpha}}{1-e^{-s}e^{-2i\alpha}}=-e^{2i\alpha}\frac{v_-}{1-v_-}+e^{-2i\alpha}\frac{v_+}{1-v_+},
\end{align}
where we have made the substitution $v_\pm=e^{s\pm 2i\alpha}$ noting that $\partial_s^{j}=(v_\pm\partial_{v_\pm} )^{j}$ under this substitution with $s=0$ corresponding to $v_\pm=e^{\pm 2i\alpha}$. Next, recall the identity (noting $\Li_1(z)=-\ln(1-z)$)
\begin{align}\label{Eq:Polylogrec}
\Li_{-n}(z)=\bigg(z\frac{\partial}{\partial z}\bigg)^n \frac{z}{1-z},\quad n\in\mathbb{N}_0,\ z\neq 1.
\end{align}
A short calculation now yields that (noting for $z=e^{2i\alpha}$, $z\partial_z=\frac{1}{2i}\partial_\alpha$ and $z/(1-z)=-\frac{1}{2}+\frac{i}{2}\cot(\alpha)$)
\begin{align}\label{Eq:Polylogid}
\Li_{-j}(e^{2i\alpha})=\frac{i^{-j}}{2^{j}}\frac{\partial^{j}}{\partial\alpha^{j}}\frac{e^{2i\alpha}}{1-e^{2i\alpha}}=i\frac{i^{-j}}{2^{j+1}}\frac{\partial^{j}}{\partial\alpha^{j}}\cot(\alpha),\quad j\in\mathbb{N},\ \alpha\in(0,\pi),
\end{align}
showing $\Li_{-j}\big(e^{2i\alpha}\big)$ is purely imaginary whenever $j\geq1$ is even, and purely real whenever $j\geq1$ is odd.
Thus we can conclude that the coefficients for the square of the function $Q_\alpha(s)$ can be written as
\begin{align}\label{Eq:akmid}
a_k(\alpha)=\frac{i}{2\sin(2\alpha) k!}\big[e^{2i\alpha}\Li_{-k}\big(e^{-2i\alpha}\big)-e^{-2i\alpha}\Li_{-k}\big(e^{2i\alpha}\big)\big]=\frac{1}{\sin(2\alpha) k!}\Im\big[e^{-2i\alpha}\Li_{-k}\big(e^{2i\alpha}\big)\big],
\end{align}
where we have used that for real $s$ and $z$ not on the branch cut $[1,\infty)$ one has that $\overline{\Li_s(z)}=\Li_s(\overline{z})$.
Thus separating by parity in $k$ leads to
\begin{equation}\label{Eq:ak}
a_k(\alpha)=\begin{cases}
-\frac{1}{k!}\Li_{-k}\big(e^{2i\alpha}\big),& k\geq1 \text{ odd},\\
\frac{-i\cot(2\alpha)}{ k!}\Li_{-k}\big(e^{2i\alpha}\big),& k\geq1 \text{ even},
\end{cases}=-[i\cot(2\alpha)]^{\frac{1+(-1)^k}{2}}\frac{1}{k!}\Li_{-k}\big(e^{2i\alpha}\big),\quad k\in\mathbb{N}.
\end{equation}
Returning to our original square root function $Q_\alpha(s)$, we have by the Cauchy product,
\begin{equation}
a_n(\alpha)=\sum_{k=0}^{n}c_k(\alpha) c_{n-k}(\alpha)\Longrightarrow c_0(\alpha)=\frac{1}{2\sin(\alpha)},\quad c_n(\alpha)=\sin(\alpha)\bigg(a_n(\alpha)-\sum_{k=1}^{n-1}c_k(\alpha) c_{n-k}(\alpha)\bigg),\quad n\in\mathbb{N},
\end{equation}
finishing the proof.
\end{proof}

We are now in a position to complete the calculation of residues.

\begin{proof}[Proof of Theorem \ref{Thm:Main2} (Residues).]
First note that the cases $x=y\in\{-1,1\}$ are covered by Lemma \ref{Lem:Spec} once again. So we now restrict ourselves to the open interval $x=y\in(-1,1)$.

To find the residues in this most general setting, we need to track the series coefficients a bit more carefully to find the overall contribution from $\widetilde{h}_j$, which will then yield the overall residue of $h_j$ and hence $K_2$.
We begin with the analysis of the hypergeometric function in $\widetilde{h}_j$ as $s\downarrow0$.
Since we are only considering residues and the poles occur only from the $\ln(s)$ term, from \eqref{Eq:2f1} we need only consider for each $k\in\mathbb{N}_0$,
\begin{equation}\label{Eq:35}
-\frac{1}{\pi} \frac{(1/2)_k^2}{(k!)^2}(1-H_{\alpha,\alpha}(s))^k[2\ln(s)]=-\frac{2\ln(s)}{\pi} \frac{(1/2)_k^2}{(k!)^2}\sum_{j=k}^\infty d_{j,k}(\alpha) s^{2j},
\end{equation}
with $d_{j,k}$ as in Lemma \ref{Lem:Hkexp} for $k\geq 1$ and $d_{j,0}(\alpha)=\delta_{j,0}$. The multiple of $2$ on the log term stems from
\begin{align}
    \ln(1-q)=\ln(1-H_{\alpha,\alpha}(s))=\ln\Big(-\sum_{j\in\mathbb{N}} d_{j,1}(\alpha)s^{2j}\Big) =2\ln(s)+\ln\Big(-\sum_{j\in\mathbb{N}} d_{j,1}(\alpha)s^{2j-2}\Big).
\end{align}
Multiplying \eqref{Eq:35} by the power series of the square root from \eqref{Eq:Sqrtseries}
and $s^{t+j-1}$, then using Cauchy product while factoring out $s^{2k}$ yields the terms multiplying the logarithmic term:
\begin{align}\label{Eq:2F1Series2}
-\frac{2\ln(s)s^{t+j-1}}{\pi} \frac{(1/2)_k^2}{(k!)^2}\bigg(\sum_{j=k}^\infty d_{j,k}(\alpha) s^{2j}\bigg)\bigg(\sum_{\ell\in\mathbb{N}_0}c_\ell(\alpha)s^\ell\bigg)=-\frac{2\ln(s)s^{t+j-1+2k}}{\pi} \frac{(1/2)_k^2}{(k!)^2}\sum_{i\in\mathbb{N}_0} b_{i,k}(\alpha) s^{i},&\notag \\
b_{i,k}(\alpha)=\sum_{j=0}^{\lfloor i/2\rfloor} d_{k+j,k}(\alpha)c_{i-2j}(\alpha),\quad \alpha\in(0,\pi),\ k\in\mathbb{N}_0.&
\end{align}
(Notice that for $k=0$, $d_{j,0}(\alpha)=\delta_{j,0}$, and this reduces to the series of $Q_\alpha(s)$ as expected.)

Now, this holds for each $k\in\mathbb{N}_0$ coming from the expansion for the hypergeometric function in \eqref{Eq:2f1}.
In particular, the logarithmic part of the integrand in \eqref{Eq:hint} (including the Gamma prefactor) after setting $\alpha=\beta$ and integrating near zero as in \eqref{Eq:logpole} will be of the form
\begin{equation}
\frac{2}{\pi\Gamma(t+j)} \sum_{k\in\mathbb{N}_0}\frac{(1/2)_k^2}{(k!)^2}\sum_{i\in\mathbb{N}_0}  \frac{b_{i,k}(\alpha)}{(i+t+j+2k)^2},\quad \alpha\in(0,\pi),\ j\in\mathbb{N}_0.
\end{equation}
So we need to consider all combinations of $k,i\in\mathbb{N}_0$ such that $2k+i=\ell\in\mathbb{N}_0$ to find the contribution to the residue at $t+j=-\ell$ from $\widetilde{h}_j$ for each $\ell\in\mathbb{N}_0$:
\begin{equation}
\frac{2(-1)^{\ell}\ell!}{\pi(t+j+\ell)} \sum_{k=0}^{\lfloor\ell/2\rfloor} \frac{(1/2)_k^2}{(k!)^2}b_{\ell-2k,k}(\alpha),
\end{equation}
where the prefactor $(-1)^{\ell}\ell!$ (and reduction of the pole) comes from behavior near $-\ell$ of the reciprocal gamma function.

Now we will have to once again recall that $h_j(t,x,y)=\widetilde{h}_{j-1}(t,x,y)+\tfrac{1}{2}\widetilde{h}_j(t,x,y)$ to find the contribution from $h_j$ itself.
In particular, one needs to be careful with accounting the contributions for different values of $j$ in the sum over $j$ in \eqref{Eq:6a}.
The full contribution from $h_j$ for the possible simple poles at $t=-m$, $m\in\mathbb{N}_0$ comes from a careful consideration of the indices, resulting in (after multiplying by the prefactor $e^{-\gamma_E t}$ and $p_{j,S_2}(t)$ from \eqref{Eq:6a}) for $\alpha\in(0,\pi),\ m\in\mathbb{N}_0$,
\begin{align}\label{Eq:GeneralPole}
\frac{2e^{-\gamma_E t}}{\pi(t+m)}\bigg[\sum_{\ell=1}^{m+1}(-1)^{\ell}\ell! p_{m+1-\ell,S_2}(t)\sum_{k=0}^{\lfloor\ell/2\rfloor} \frac{(1/2)_k^2}{(k!)^2}b_{\ell-2k,k}(\alpha)+\frac{1}{2}\sum_{\ell=0}^{m}(-1)^{\ell}\ell! p_{m-\ell,S_2}(t)\sum_{k=0}^{\lfloor\ell/2\rfloor} \frac{(1/2)_k^2}{(k!)^2}b_{\ell-2k,k}(\alpha)\bigg].
\end{align}
This yields the residues at the possible simple poles $t=-m$, however, there is an additional pole at $t=1$ from $\widetilde{h}_{j-1}$ with $j=0$ and $\ell=0$ above which has residue given by, since $p_{0,S_2}(t)=1$,
\begin{equation}
\frac{2e^{-\gamma_E}}{\pi}p_{0,S_2}(1)b_{0,0}(\alpha)=\frac{2e^{-\gamma_E}}{\pi}c_{0}(\alpha)=\frac{e^{-\gamma_E}}{\pi\sin(\alpha)}.
\end{equation}
Finally, note that $t=m=0$ is not a pole since the residue in \eqref{Eq:GeneralPole} becomes
\begin{align}
&-\frac{2}{\pi} p_{0,S_2}(0) b_{1,0}(\alpha)+\frac{1}{\pi} p_{0,S_2}(0) b_{0,0}(\alpha)=-\frac{2}{\pi}c_1(\alpha)+\frac{1}{\pi}c_0(\alpha)=-\frac{2}{4\pi\sin(\alpha)}+\frac{1}{2\pi\sin(\alpha)}=0,
\end{align}
completing the proof of the structure of residues.
\end{proof}

\appendix

\section{Some identities regarding Bell polynomials and polylogarithms}\label{AppPoly}

In this appendix, we offer a different way of expanding the factor $(1-H_{\alpha,\alpha}(s))^k$ needed when calculating the residues above via Bell polynomials. When comparing to the prior method, this results in some interesting identities regarding Bell polynomials with polylogarithm entries.

\begin{prop}\label{Lem:H}
The following holds for $\alpha\in(0,\pi)$ and $s\geq0$: 
\begin{align}\label{Eq:Hseries}
&H_{\alpha,\alpha}(s)=\frac{4e^{-s}\sin^2(\alpha)}{1-2e^{-s}\cos(2\alpha)+e^{-2s}}=1+\sum_{j\in\mathbb{N}} C_j(\alpha) s^{2j},\ C_j(\alpha)=\begin{cases}
-i\frac{2\tan(\alpha)}{(2j)!}\Li_{-2j}(e^{2i\alpha}),& \hspace{-.1cm}\alpha\in(0,\pi)\backslash\{\tfrac{\pi}{2}\},\\[2mm]
\frac{2(2^{2j+2}-1)B_{2j+2}}{(j+1)(2j)!},& \hspace{-.1cm}\alpha=\tfrac{\pi}{2}.
\end{cases}
\end{align}
\end{prop}

\begin{remark}\label{Rem:H}
Note the value for $\alpha=\tfrac{\pi}{2}$ can be found by taking the limit in the first line. We also point out that using standard identities for the polylogarithm (e.g. combining \eqref{eq:2.49} and Exercise 5.2 in \cite{Co74}), $C_j(\alpha)$ can also be written as 
\begin{align}
C_j(\alpha)=-i\frac{2\tan(\alpha)}{(2j)!}\sum_{n=0}^{2j} n! S(2j+1,n+1)\bigg(-\frac{1}{2}+\frac{i}{2}\cot(\alpha)\bigg)^{n+1},\quad \alpha\in(0,\pi)\backslash\{\tfrac{\pi}{2}\},
\end{align}
where $S(n,k)$ are Stirling numbers of the second kind. Moreover, utilizing Bell polynomials leads to 
\begin{equation}
C_j(\alpha)=\sum_{n=0}^{2j}\frac{(-1)^n}{4^n\sin^{2n}(\alpha)}\widehat{B}_{2j,n}(w_1,\dots,w_{2j-n+1}),\quad w_{\ell}=\frac{1+(-1)^\ell}{\ell!},\quad \ell\in\mathbb{N},
\end{equation}
which yields interesting identities for these sums of Bell polynomials when compared to the previous forms.
\end{remark}

\begin{proof}[Proof of Proposition \ref{Lem:H}.]
First note that $H_{\alpha,\alpha}(s)$ is an even function of $s$. We begin with the case $\alpha=\tfrac{\pi}{2}$, which must be treated separately, but also has a closed trigonometric form:
\begin{equation}
H_{\tfrac{\pi}{2},\tfrac{\pi}{2}}(s)=\sech^2(\tfrac{s}{2})=1+\sum_{j\in\mathbb{N}} C_j(\tfrac{\pi}{2}) s^{2j},\quad C_j(\tfrac{\pi}{2})=\frac{2(2^{2j+2}-1)B_{2j+2}}{(j+1)(2j)!}.
\end{equation}

Turning to the general case $\alpha\in(0,\pi)\backslash\{\tfrac{\pi}{2}\}$, we proceed similarly to the proof of Lemma \ref{Lem:Sqrt}:
\begin{align}
& H_{\alpha,\alpha}(s)=\frac{4e^{-s}\sin^2(\alpha)}{1-2e^{-s}\cos(2\alpha)+e^{-2s}}=4e^{-s}\sin^2(\alpha)\frac{1}{(1-e^{-s}e^{2i\alpha})(1-e^{-s}e^{-2i\alpha})}\\
&\quad =\frac{4e^{-s}\sin^2(\alpha)}{2i\sin(2\alpha)}\bigg(\frac{e^{2i\alpha}}{1-e^{-s}e^{2i\alpha}}-\frac{e^{-2i\alpha}}{1-e^{-s}e^{-2i\alpha}}\bigg)=-i\tan(\alpha)\bigg(\frac{u_+}{1-u_+}-\frac{u_-}{1-u_-}\bigg),\notag
\end{align}
where we have made the substitution $u_\pm =e^{\pm 2i\alpha}e^{-s}$ since $\partial_s^{2j}=(u_\pm\partial_{u_\pm} )^{2j}$ under this substitution with $s=0$ corresponding to $u_\pm=e^{\pm 2i\alpha}$.
Recalling \eqref{Eq:Polylogrec}, this shows that the coefficient of $s^{2j}$ of the series expansion of $H_{\alpha,\alpha}(s)$ about $s=0$ for $j\in\mathbb{N}$ amounts to considering
\begin{align}
(u_\pm\partial_{u_\pm})^{2j}\frac{u_\pm}{1-u_\pm}\bigg|_{u_\pm =e^{\pm 2i\alpha}}=\Li_{-2j}(e^{\pm 2i\alpha}),\quad j\in\mathbb{N}.
\end{align}
Noticing once again that for real $s$ and $z$ not on the branch cut $[1,\infty)$ one has that $\overline{\Li_s(z)}=\Li_s(\overline{z})$, we are able to conclude that the coefficient of $s^{2j}$ is given by
\begin{equation}
\frac{-i\tan(\alpha)}{(2j)!}(2i)\Im\big(\Li_{-2j}(e^{2i\alpha})\big)=\frac{2\tan(\alpha)}{(2j)!}\Im\big(\Li_{-2j}(e^{2i\alpha})\big)=-i\frac{2\tan(\alpha)}{(2j)!}\Li_{-2j}(e^{2i\alpha}),\quad j\in\mathbb{N},
\end{equation}
where the last equality follows from \eqref{Eq:Polylogid}.
\end{proof}

We now offer an alternate way of analyzing the series coefficients of $(1-H_{\alpha,\alpha}(s))^k$. 
In this direction, it will be helpful to first note that expanding the left-hand side of \eqref{Eq:BellGena} in a power series and taking the coefficient of $u^k$ on both sides yields the identity
\begin{equation}
\bigg(\sum_{\ell\in\mathbb{N}} x_\ell y^\ell\bigg)^k=\sum_{j=k}^\infty \widehat{B}_{j,k}(x_1,\dots,x_{j-k+1})y^j.
\end{equation}
This allows us to write the following series representation via the expansion found in Proposition \ref{Lem:H}:
\begin{equation}
(1-H_{\alpha,\alpha}(s))^k=\bigg(-\sum_{j\in\mathbb{N}}\,C_j(\alpha) s^{2j}\bigg)^k=\sum_{\ell=k}^\infty (-1)^{k}\widehat{B}_{\ell,k}(C_1(\alpha),\dots,C_{\ell-k+1}(\alpha)) s^{2\ell}.
\end{equation}
By \eqref{Eq:BellEq} and Lemma \ref{Lem:Hkexp}, this yields the identity for $j\geq k$, $k\in\mathbb{N}$, and $\alpha\in(0,\pi)\backslash\{\tfrac{\pi}{2}\}$,
\begin{equation}
i^k 2^k \tan^k(\alpha)\sum_{j_1+\cdots+j_k=j}\prod_{n=1}^k \frac{\Li_{-2j_n}(e^{2i\alpha})(\alpha)}{(2j_n)!}=\frac{1}{(2j)!}\sum_{\ell=k}^{j}\binom{\ell-1}{k-1}\frac{(-1)^{\ell-k}}{4^\ell \sin^{2\ell}(\alpha)}
\sum_{m=0}^{2\ell}
(-1)^m
\binom{2\ell}{m}
(\ell-m)^{2j}.
\end{equation}
The special case $\alpha=\tfrac{\pi}{2}$, which can be understood via taking the limit once again or via \eqref{Eq:Hseries}, yields
\begin{equation}
2^k\sum_{j_1+\cdots+j_k=j}\prod_{n=1}^k \frac{(1-2^{2j_n+2})B_{2j_n+2}}{(j_n+1)(2j_n)!}=\frac{1}{(2j)!}\sum_{\ell=k}^{j}\binom{\ell-1}{k-1}\frac{(-1)^{\ell-k}}{4^\ell}
\sum_{m=0}^{2\ell}
(-1)^m
\binom{2\ell}{m}
(\ell-m)^{2j}.
\end{equation}

\medskip

\noindent {\bf Acknowledgments.}
We thank Ovidiu Costin, Guglielmo Fucci, Vincent Martinez, and Mateusz Piorkowski for feedback on earlier forms of this paper. We are greatly indebted to Daniel Herden for pointing out the references \cite{E13,Ya12}. JS was supported in part by an AMS--Simons Travel Grant.


\begin{thebibliography}{99}
\bibitem{CEV13} C. P. Chen, N.\ Elezovi\'{c}, and L. Vuk\v{s}i\'{c}, \textit{Asymptotic formulae associated with the Wallis power function and digamma function}, J. Class. Anal. \textbf{2}(2), 151--166 (2013).

\bibitem{CV24} H.\ Chen and L.\ Véron, \textit{The Cauchy problem associated to the logarithmic Laplacian with an application to the fundamental solution}, J. Funct. Anal. {\bf 287}, 3, 72 pp (2024).

\bibitem{CW19} H.\ Chen and T.\ Weth, \textit{The Dirichlet problem for the logarithmic Laplacian}, Comm. Partial Differential Equations {\bf 44}, 11,  1100--1139 (2019).

\bibitem{Co74} L. Comtet, \textit{Advanced Combinatorics: The Art of Finite and Infinite Expansions}, Reidel Publishing Company, Dordrecht/Boston, 1974.

\bibitem{DS13} G.\ V.\ Dunne and C. Schubert, \textit{Bernoulli number identities from quantum field theory and topological string theory}, Commun. Number Theory Phys. {\bf 7}, 2, 225--249. (2013). 

\bibitem{E13} N.\ Elezovi\'{c}, \textit{Asymptotic expansions of exponentials of digamma function and identity for Bernoulli polynomials}, preprint \arxiv{1312.1604}.

\bibitem{FeS25} J.\ C.\ Fernández and A.\ Saldaña, {\it The conformal logarithmic Laplacian on the sphere:
Yamabe-type problems and Sobolev spaces}, preprint \arxiv{2507.21779}.

\bibitem{FPS25} G.\ Fucci,  M.\ Piorkowski, and J.\ Stanfill, {\it The spectral $\zeta$-function for quasi-regular Sturm--Liouville operators}, Lett. Math. Phys. {\bf 115}, 8, 48 pp. (2025).

\bibitem{FS25} G.\ Fucci and J.\ Stanfill, \textit{The exotic structure of the spectral $\zeta$-function for the Schr\"odinger operator with P\"oschl--Teller potential}, Ann. H. Poincar\'e, 44 pp. (2025). \doi{10.1007/s00023-025-01587-7}

\bibitem{GS19} P.\ Garbaczewski and V.\ Stephanovich, \textit{Fractional Laplacians in bounded domains: Killed, reflected, censored and taboo Lévy
flights}, Phys. Rev. E {\bf 99}, 4, 22 pp. (2019).

\bibitem{GM25} D.\ M.\ Gerontogiannis and B. Mesland, \textit{The logarithmic Dirichlet Laplacian on Ahflors regular spaces}, Trans. Amer. Math. Soc. {\bf 378}, 1, 651--678 (2025).

\bibitem{G05} I.\ Gessel, \textit{On Miki's identity for Bernoulli numbers}, J. Number Theory {\bf 110}, 1, 75--82 (2005).

\bibitem{LW21} A.\ Laptev and T.\ Weth, \textit{Spectral properties of the logarithmic Laplacian}, Anal. Math. Phys. {\bf 11}, 133, 24 pp. (2021).

\bibitem{Luke69} Y.L.\ Luke, \textit{The Special Functions and Their Approximations, vol. I}, Academic Press, New York, 1969.

\bibitem{Max56} L.C.\ Maximon, \textit{A generating function for the product of two Legendre polynomials}, Norske Vid. Selsk. Forh., Trondheim \textbf{29}, 5 pp. (1956).

\bibitem{Maz18} V. Maz'ya, \textit{Seventy Five (Thousand) Unsolved Problems in Analysis and Partial Differential Equations}, Integr. Equ. Oper. Theory \textbf{90}, 25 (2018).

\bibitem{MNP00} V. Maz'ya, S. Nazarov and B. Plamenevskij, \textit{Asymptotic Theory of Elliptic Boundary Value Problems in Singularly Perturbed Domains, vol. II}, Springer, Basel, 2000.

\bibitem{DLMF} W. J. Olver, A. B. Olde Daalhuis, D. W. Lozier, B. I. Schneider, R. F. Boisvert, C. W. Clark, B. R. Miller, B. V. Saunders, H. S. Cohl, and M. A. McClain (eds.), {\it NIST Digital Library of Mathematical Functions}, http://dlmf.nist.gov/, Release 1.1.12 of 2023-12-15.

\bibitem{OS22} C. O'Sullivan, \textit{De Moivre and Bell polynomials}, Expo. Math. \textbf{40}, 4, 23 pp. (2022).

\bibitem{Sz39} G. Szegö, \textit{Orthogonal Polynomials}, American Mathematical Society, Providence, RI, 1939.

\bibitem{T64} E.O. Tuck, \textit{Some methods for flows past blunt slender bodies}, J. Fluid Mech. \textbf{18} (4), 619--635 (1964).

\bibitem{Ya12} S. Yang, \textit{On an open problem of Chen and Mortici concerning the
Euler–Mascheroni constant}, J. Math. Anal. Appl. \textbf{396}, 689--693 (2012). 

\bibitem{Zag14} D. Zagier, \textit{Curious and exotic identities for Bernoulli
numbers} in T. Arakawa, T. Ibukiyama, M. Kaneko, \textit{Bernoulli Numbers and Zeta Functions}, Springer, Basel, 2014.
\end{thebibliography}
\end{document}